\documentclass[a4paper,10pt, final]{amsart}
\usepackage[utf8x]{inputenc}
\usepackage{amsmath, amsthm, amsfonts,amssymb,epsfig,enumerate, wrapfig, tikz , scalerel, color, tensor,comment, todo, stmaryrd, caption}
\usepackage{xcolor,colortbl}
\usetikzlibrary{decorations.markings}
\newcounter{res}[section]
\numberwithin{res}{section}
\newtheorem{thm}[res]{Theorem}
\newtheorem{lem}[res]{Lemma}

\newtheorem*{claim*}{Claim}
\newtheorem{prop}[res]{Proposition}

\theoremstyle{definition}

\newtheorem{dfn}[res]{Definition}

\newtheorem{rmk}[res]{Remark}

\newcommand{\NN}{\mathbb N} 
\newcommand{\ZZ}{\mathbb Z} 
\newcommand{\CC}{\mathbb C} 

\newcommand{\RR}{\mathbb R} 

\renewcommand {\epsilon}{\varepsilon}

\renewcommand {\leq}{\leqslant}
\renewcommand {\geq}{\geqslant}

\newcommand{\qbin}[2]{\ensuremath
\begin{pmatrix}
  #1 + #2 \\
#1 \quad #2
\end{pmatrix}_{\!\!q}
}
\newcommand{\qbinb}[3]{\ensuremath
\begin{pmatrix}
  #1 \\
#2 \quad #3
\end{pmatrix}_{\!\!q}
}

\let\oldmarginpar\marginpar
\renewcommand\marginpar[1]{\oldmarginpar{\color{red}\raggedleft\bf #1}}

\newcommand\kups[1]{\left\langle #1 \right\rangle}
\newcommand\kup[1]{\ensuremath{{\stretchleftright {\bigg\langle }#1 {\bigg\rangle}}}}
\newcommand\kupc[1]{\ensuremath{{\stretchleftright {\bigg\langle }#1 {\bigg\rangle}}_{\mathrm{col}}}}
\newcommand\kupcs[1]{\ensuremath{\left\langle #1 \right\rangle}_{\mathrm{col}}}

\let\oldtodo\todo
\renewcommand{\todo}[1]{{\color{blue} #1}\oldtodo[]{#1}}

\newcommand{\imagesfolder}{.}

\newcommand{\sll}{\ensuremath{\mathfrak{sl}}}

\title{A new way to evaluate MOY graphs}
\author{Louis-Hadrien Robert}
\address{Universit\"a{}t Hamburg \\ Bundesstra\ss{}e 55\\ 20146 Hamburg \\ Germany }
\email{louis-hadrien.robert@uni-hamburg.de}

\newcommand{\Y}{\ensuremath{\rotatebox[origin=c]{180}{\textnormal{\textsf{Y}}}}}

\begin{document}
\tikzset{>=latex}
\newcommand{\digona}{\ensuremath{\vcenter{\hbox{\tikz[scale=0.4]{
\coordinate (B) at (0,0);
\coordinate (V1) at (0,0.5);
\coordinate (V2) at (0,2.5);
\coordinate (T) at (0,3);
\draw[white] (0, -0.5) -- (0, 3.5);
\draw[->] (B) -- (V1) node[midway, left] {\tiny{$m+n$}};
\draw[->] (V2) -- (T) node[midway, left] {\tiny{$m+n$}};
\draw[->] (V1) .. controls +(+0.5, +0.5) and +(+0.5, -0.5).. (V2) node[midway, right] {\tiny{$n$}};
\draw[->] (V1) .. controls +(-0.5, +0.5) and +(-0.5, -0.5).. (V2) node[midway, left] {\tiny{$m$}};
}}}}}

\newcommand{\digonNk}{\ensuremath{\vcenter{\hbox{\tikz[scale=0.4]{
\coordinate (B) at (0,0);
\coordinate (V1) at (0,0.5);
\coordinate (V2) at (0,2.5);
\coordinate (T) at (0,3);
\draw[white] (0, -0.5) -- (0, 3.5);
\draw[->] (B) -- (V1) node[midway, left] {\tiny{$N$}};
\draw[->] (V2) -- (T) node[midway, left] {\tiny{$N$}};
\draw[->] (V1) .. controls +(+1, +0.5) and +(+1, -0.5).. (V2) node[midway, right] {\tiny{$k$}};
\draw[->] (V1) .. controls +(-1, +0.5) and +(-1, -0.5).. (V2) node[midway, left] {\tiny{$N-k$}};
}}}}}

\newcommand{\verta}{\ensuremath{\vcenter{\hbox{\tikz[scale=0.4]{
\coordinate (B) at (0,0);
\coordinate (T) at (0,3);
\draw[white] (0, -0.5) -- (0, 3.5);
\draw[->] (B) -- (T) node[midway, right] {\tiny{$m+n$}};
}}}}}

\newcommand{\vertN}{\ensuremath{\vcenter{\hbox{\tikz[scale=0.4]{
\coordinate (B) at (0,0);
\coordinate (T) at (0,3);
\draw[white] (0, -0.5) -- (0, 3.5);
\draw[->] (B) -- (T) node[midway, right] {\tiny{$N$}};
}}}}}

\newcommand{\digonb}{\ensuremath{\vcenter{\hbox{\tikz[scale=0.4]{
\coordinate (B) at (0,0);
\coordinate (V1) at (0,0.5);
\coordinate (V2) at (0,2.5);
\coordinate (T) at (0,3);
\draw[white] (0, -0.5) -- (0, 3.5);
\draw[->] (B) -- (V1) node[midway, left] {\tiny{$m$}};
\draw[->] (V2) -- (T) node[midway, left] {\tiny{$m$}};
\draw[<-] (V1) .. controls +(+0.5, +0.5) and +(+0.5, -0.5).. (V2) node[midway, right] {\tiny{$n$}};
\draw[->] (V1) .. controls +(-0.5, +0.5) and +(-0.5, -0.5).. (V2) node[midway, left] {\tiny{$m+n$}};
}}}}}

\newcommand{\vertb}{\ensuremath{\vcenter{\hbox{\tikz[scale=0.4]{
\coordinate (B) at (0,0);
\coordinate (T) at (0,3);
\draw[white] (0, -0.5) -- (0, 3.5);
\draw[->] (B) -- (T) node[midway, right] {\tiny{$m$}};
}}}}}

\newcommand{\vertc}{\ensuremath{\vcenter{\hbox{\tikz[scale=0.4]{
\coordinate (B) at (0,0);
\coordinate (T) at (0,3);
\draw[white] (0, -0.5) -- (0, 3.5);
\draw[<-] (B) -- (T) node[midway, right] {\tiny{$N-m$}};
}}}}}

\newcommand{\digonc}{\ensuremath{\vcenter{\hbox{\tikz[scale=0.4]{
\coordinate (B) at (0,0);
\coordinate (V1) at (0,0.5);
\coordinate (V2) at (0,2.5);
\coordinate (T) at (0,3);
\draw[white] (0, -0.5) -- (0, 3.5);
\draw[<-] (B) -- (V1) node[midway, left] {\tiny{$N-m$}};
\draw[<-] (V2) -- (T) node[midway, left] {\tiny{$N-m$}};
\draw[<-] (V1) .. controls +(+0.5, +0.5) and +(+0.5, -0.5).. (V2) node[midway, right] {\tiny{$n$}};
\draw[<-] (V1) .. controls +(-0.5, +0.5) and +(-0.5, -0.5).. (V2) node[midway, left] {\tiny{$N-(m+n)$}};
}}}}}

\newcommand{\stemptya}{
\draw[white, opacity = 0] (0, -0.5) -- (0, 3.5);
\coordinate (B) at (0,0);
\coordinate (V1) at (0,1);
\coordinate (V2) at (1,2);
\coordinate (T1) at (-2,3);
\coordinate (T2) at (0,3);
\coordinate (T3) at (2,3);
\coordinate (la) at (0, -1);
\draw[dotted] (B)-- (V1) --(T1);
\draw[dotted] (T2)-- (V2) --(T3);
\draw[dotted] (V2) --(V1);
}

\newcommand{\stemptyb}{
\draw[white, opacity=0] (0, -0.5) -- (0, 3.5);
\coordinate (B) at (0,0);
\coordinate (V1) at (0,1);
\coordinate (V2) at (-1,2);
\coordinate (T1) at (-2,3);
\coordinate (T2) at (0,3);
\coordinate (T3) at (2,3);
\coordinate (la) at (0, -1);
\draw[dotted] (B)-- (V1) --(T3);
\draw[dotted] (T1)-- (V2) --(T2);
\draw[dotted] (V2) --(V1);
}

\newcommand{\stnoa}{\ensuremath{\vcenter{\hbox{\tikz[scale=0.2]{
\stemptya
\node at (la) {$0$};
}}}}}

\newcommand{\stnob}{\ensuremath{\vcenter{\hbox{\tikz[scale=0.2]{
\stemptyb
\node at (la) {$0$};
}}}}}

\newcommand{\stonea}{\ensuremath{\vcenter{\hbox{\tikz[scale=0.2]{
\stemptya
\draw[->] (B) -- (V1) --(T1); 
\node at (la) {$0$};
}}}}}

\newcommand{\stoneb}{\ensuremath{\vcenter{\hbox{\tikz[scale=0.2]{
\stemptyb
\draw[->] (B) -- (V1) --(T1); 
\node at (la) {$0$};
}}}}}

\newcommand{\sttwoa}{\ensuremath{\vcenter{\hbox{\tikz[scale=0.2]{
\stemptya
\draw[->] (B) -- (V1) -- (V2)--(T2); 
\node at (la) {$0$};
}}}}}

\newcommand{\sttwob}{\ensuremath{\vcenter{\hbox{\tikz[scale=0.2]{
\stemptyb
\draw[->] (B) -- (V1) -- (V2) --(T2); 
\node at (la) {$0$};
}}}}}

\newcommand{\stthreea}{\ensuremath{\vcenter{\hbox{\tikz[scale=0.2]{
\stemptya
\draw[->] (B) -- (V1) --(T3); 
\node at (la) {$0$};
}}}}}

\newcommand{\stthreeb}{\ensuremath{\vcenter{\hbox{\tikz[scale=0.2]{
\stemptyb
\draw[->] (B) -- (V1) --(T3); 
\node at (la) {$0$};
}}}}}

\newcommand{\stdonea}{\ensuremath{\vcenter{\hbox{\tikz[scale=0.2]{
\stemptya
\draw[<-] (B) -- (V1) --(T1); 
\node at (la) {$0$};
}}}}}

\newcommand{\stdoneb}{\ensuremath{\vcenter{\hbox{\tikz[scale=0.2]{
\stemptyb
\draw[<-] (B) -- (V1) --(T1); 
\node at (la) {$0$};
}}}}}

\newcommand{\stdtwoa}{\ensuremath{\vcenter{\hbox{\tikz[scale=0.2]{
\stemptya
\draw[<-] (B) -- (V1) -- (V2)--(T2); 
\node at (la) {$0$};
}}}}}

\newcommand{\stdtwob}{\ensuremath{\vcenter{\hbox{\tikz[scale=0.2]{
\stemptyb
\draw[<-] (B) -- (V1) -- (V2) --(T2); 
\node at (la) {$0$};
}}}}}

\newcommand{\stdthreea}{\ensuremath{\vcenter{\hbox{\tikz[scale=0.2]{
\stemptya
\draw[<-] (B) -- (V1) --(T3); 
\node at (la) {$0$};
}}}}}

\newcommand{\stdthreeb}{\ensuremath{\vcenter{\hbox{\tikz[scale=0.2]{
\stemptyb
\draw[<-] (B) -- (V1) --(T3); 
\node at (la) {$0$};
}}}}}

\newcommand{\stonetwoa}{\ensuremath{\vcenter{\hbox{\tikz[scale=0.2]{
\stemptya
\draw[->] (T1) -- (V1) --(V2) --(T2); 
\node at (la) {$+\frac12$};
}}}}}

\newcommand{\stonetwob}{\ensuremath{\vcenter{\hbox{\tikz[scale=0.2]{
\stemptyb
\draw[->] (T1) -- (V2) --(T2); 
\node at (la) {$+\frac12$};
}}}}}

\newcommand{\sttwoonea}{\ensuremath{\vcenter{\hbox{\tikz[scale=0.2]{
\stemptya
\draw[<-] (T1) -- (V1) --(V2) --(T2); 
\node at (la) {$-\frac12$};
}}}}}

\newcommand{\sttwooneb}{\ensuremath{\vcenter{\hbox{\tikz[scale=0.2]{
\stemptyb
\draw[<-] (T1) -- (V2) --(T2); 
\node at (la) {$-\frac12$};
}}}}}

\newcommand{\stonethreea}{\ensuremath{\vcenter{\hbox{\tikz[scale=0.2]{
\stemptya
\draw[->] (T1) -- (V1)  --(T3); 
\node at (la) {$+\frac12$};
}}}}}

\newcommand{\stonethreeb}{\ensuremath{\vcenter{\hbox{\tikz[scale=0.2]{
\stemptyb
\draw[->] (T1) -- (V1) --(T3); 
\node at (la) {$+\frac12$};
}}}}}

\newcommand{\stthreeonea}{\ensuremath{\vcenter{\hbox{\tikz[scale=0.2]{
\stemptya
\draw[<-] (T1) -- (V1)  --(T3); 
\node at (la) {$-\frac12$};
}}}}}

\newcommand{\stthreeoneb}{\ensuremath{\vcenter{\hbox{\tikz[scale=0.2]{
\stemptyb
\draw[<-] (T1) -- (V1) --(T3); 
\node at (la) {$-\frac12$};
}}}}}

\newcommand{\sttwothreea}{\ensuremath{\vcenter{\hbox{\tikz[scale=0.2]{
\stemptya
\draw[->] (T2) -- (V2) -- (T3); 
\node at (la) {$+\frac12$};
}}}}}

\newcommand{\sttwothreeb}{\ensuremath{\vcenter{\hbox{\tikz[scale=0.2]{
\stemptyb
\draw[->] (T2) -- (V2) -- (V1) -- (T3); 
\node at (la) {$+\frac12$};
}}}}}

\newcommand{\stthreetwoa}{\ensuremath{\vcenter{\hbox{\tikz[scale=0.2]{
\stemptya
\draw[<-] (T2) -- (V2) -- (T3); 
\node at (la) {$-\frac12$};
}}}}}

\newcommand{\stthreetwob}{\ensuremath{\vcenter{\hbox{\tikz[scale=0.2]{
\stemptyb
\draw[<-] (T2) -- (V2) -- (V1) -- (T3); 
\node at (la) {$-\frac12$};
}}}}}

\newcommand{\stgamma}{\ensuremath{\vcenter{\hbox{\tikz[scale=0.3]{
\coordinate (B) at (0,0);
\coordinate (V1) at (0,1);
\coordinate (V2) at (1,2);
\coordinate (T1) at (-2,3);
\coordinate (T2) at (0,3);
\coordinate (T3) at (2,3);
\draw[->] (B) -- (V1) node [at start, below] {\tiny{$i+j+k$}};
\draw[->] (V1) -- (T1) node [at end, above] {\tiny{$i$}};
\draw[->] (V1)  -- (V2) node[midway, right] {\tiny{$j+k$}};
\draw[->] (V2) -- (T2) node[at end, above] {\tiny{$j$}};
\draw[->] (V2) -- (T3) node[at end, above] {\tiny{$k$}};
}}}}}

\newcommand{\stgammaprime}{\ensuremath{\vcenter{\hbox{\tikz[scale=0.3]{
\coordinate (B) at (0,0);
\coordinate (V1) at (0,1);
\coordinate (V2) at (-1,2);
\coordinate (T1) at (-2,3);
\coordinate (T2) at (0,3);
\coordinate (T3) at (2,3);
\draw[->] (B) -- (V1) node [at start, below] {\tiny{$i+j+k$}};
\draw[->] (V1) -- (T3) node [at end, above] {\tiny{$k$}};
\draw[->] (V1)  -- (V2) node[midway, left] {\tiny{$i+j$}};
\draw[->] (V2) -- (T1) node[at end, above] {\tiny{$i$}};
\draw[->] (V2) -- (T2) node[at end, above] {\tiny{$j$}};
}}}}}
\newcommand{\squarea}{\ensuremath{\vcenter{\hbox{\tikz[scale=0.4]{
\coordinate (B1) at (-1,0);
\coordinate (B2) at (1,0);
\coordinate (C1) at (-1,1);
\coordinate (D1) at (-1,2);
\coordinate (C2) at (1,1);
\coordinate (D2) at (1,2);
\coordinate (T1) at (-1,3);
\coordinate (T2) at (1,3);
\draw[->] (B1) -- (C1) node[at start, below] {\tiny{$1$}};
\draw[->] (D1) -- (C1) node[midway, left   ] {\tiny{$m$}};
\draw[->] (D1) -- (T1) node[at end , above ] {\tiny{$1$}};
\draw[->] (C2) -- (B2) node[at end, below] {\tiny{$m$}};
\draw[->] (C2) -- (D2) node[midway, right] {\tiny{$1$}};
\draw[->] (T2) -- (D2) node[at start, above] {\tiny{$m$}};
\draw[->] (D2) -- (D1) node[midway, above] {\tiny{$m+1$}};
\draw[->] (C1) -- (C2) node[midway, below] {\tiny{$m+1$}};
}}}}}

\newcommand{\twoverta}{\ensuremath{\vcenter{\hbox{\tikz[scale=0.4]{
\coordinate (B1) at (-1,0);
\coordinate (T1) at (-1,3);
\coordinate (B2) at (1,0);
\coordinate (T2) at (1,3);
\draw[->] (B1) -- (T1) node[midway, left] {\tiny{$1$}};
\draw[->] (T2) -- (B2) node[midway, right] {\tiny{$m$}};
}}}}}

\newcommand{\doubleYa}{\ensuremath{\vcenter{\hbox{\tikz[scale=0.4]{
\coordinate (B1) at (-1,0);
\coordinate (T1) at (-1,3);
\coordinate (C) at (0,1);
\coordinate (D) at (0,2);
\coordinate (B2) at (1,0);
\coordinate (T2) at (1,3);
\draw[->] (B1) -- (C) node[at start, below] {\tiny{$1$}};
\draw[->] (C) -- (B2) node[at end, below] {\tiny{$m$}};
\draw[->] (D) -- (C) node[midway, left] {\tiny{$m-1$}};
\draw[->] (T2) -- (D) node[at start, above] {\tiny{$m$}};
\draw[->] (D) -- (T1) node[at end, above] {\tiny{$1$}};
}}}}}

\newcommand{\squareb}{\ensuremath{\vcenter{\hbox{\tikz[scale=0.55]{
\coordinate (B1) at (-1,0);
\coordinate (B2) at (1,0);
\coordinate (C1) at (-1,1);
\coordinate (D1) at (-1,2);
\coordinate (C2) at (1,1);
\coordinate (D2) at (1,2);
\coordinate (T1) at (-1,3);
\coordinate (T2) at (1,3);
\draw[->] (B1) -- (C1) node[at start, below] {\tiny{$1$}};
\draw[->] (C1) -- (D1) node[midway, left   ] {\tiny{$l+n$}};
\draw[->] (D1) -- (T1) node[at end , above ] {\tiny{$l$}};
\draw[->] (B2) -- (C2) node[at start, below] {\tiny{$m+l-1$}};
\draw[->] (C2) -- (D2) node[midway, right] {\tiny{$m-n$}};
\draw[->] (D2) -- (T2) node[at end, above] {\tiny{$m$}};
\draw[->] (D1) -- (D2) node[midway, above] {\tiny{$n$}};
\draw[->] (C2) -- (C1) node[midway, below] {\tiny{$l+n-1$}};
}}}}}

\newcommand{\squarec}{\ensuremath{\vcenter{\hbox{\tikz[xscale=0.65, yscale=0.55]{
\coordinate (B1) at (-1,0);
\coordinate (B2) at (1,0);
\coordinate (C1) at (-1,1);
\coordinate (D1) at (-1,2);
\coordinate (C2) at (1,1);
\coordinate (D2) at (1,2);
\coordinate (T1) at (-1,3);
\coordinate (T2) at (1,3);
\draw[->] (B1) -- (C1) node[at start, below] {\tiny{$n$}};
\draw[->] (C1) -- (D1) node[midway, left   ] {\tiny{$n+k $}};
\draw[->] (D1) -- (T1) node[at end , above ] {\tiny{$m$}};
\draw[->] (B2) -- (C2) node[at start, below] {\tiny{$m+l$}};
\draw[->] (C2) -- (D2) node[midway, right] {\tiny{$m+l-k$}};
\draw[->] (D2) -- (T2) node[at end, above] {\tiny{$n+l$}};
\draw[->] (D1) -- (D2) node[midway, above] {\tiny{$n+k-m$}};
\draw[->] (C2) -- (C1) node[midway, below] {\tiny{$k$}};
}}}}}

\newcommand{\squared}{\ensuremath{\vcenter{\hbox{\tikz[yscale=0.55, xscale=0.65]{
\coordinate (B1) at (-1,0);
\coordinate (B2) at (1,0);
\coordinate (C1) at (-1,1);
\coordinate (D1) at (-1,2);
\coordinate (C2) at (1,1);
\coordinate (D2) at (1,2);
\coordinate (T1) at (-1,3);
\coordinate (T2) at (1,3);
\draw[->] (B1) -- (C1) node[at start, below] {\tiny{$n$}};
\draw[->] (C1) -- (D1) node[midway, left   ] {\tiny{$m-j$}};
\draw[->] (D1) -- (T1) node[at end , above ] {\tiny{$m$}};
\draw[->] (B2) -- (C2) node[at start, below] {\tiny{$m+l$}};
\draw[->] (C2) -- (D2) node[midway, right] {\tiny{$n+l+j$}};
\draw[->] (D2) -- (T2) node[at end, above] {\tiny{$n+l$}};
\draw[->] (D2) -- (D1) node[midway, above] {\tiny{$j$}};
\draw[->] (C1) -- (C2) node[midway, below] {\tiny{$n+j-m$}};
}}}}}

\newcommand{\doubleYb}{\ensuremath{\vcenter{\hbox{\tikz[scale=0.4]{
\coordinate (B1) at (-1,0);
\coordinate (T1) at (-1,3);
\coordinate (C) at (0,1);
\coordinate (D) at (0,2);
\coordinate (B2) at (1,0);
\coordinate (T2) at (1,3);
\draw[->] (B1) -- (C) node[at start, below] {\tiny{$1$}};
\draw[<-] (C) -- (B2) node[at end, below] {\tiny{$m+l-1$}};
\draw[<-] (D) -- (C) node[midway, left] {\tiny{$l+m$}};
\draw[<-] (T2) -- (D) node[at start, above] {\tiny{$m$}};
\draw[->] (D) -- (T1) node[at end, above] {\tiny{$l$}};
}}}}}

\newcommand{\bigHb}{\ensuremath{\vcenter{\hbox{\tikz[scale=0.4]{
\coordinate (B1) at (-1,0);
\coordinate (T1) at (-1,3);
\coordinate (M1) at (-1,1.5);
\coordinate (M2) at (1,1.5);
\coordinate (B2) at (1,0);
\coordinate (T2) at (1,3);
\draw[->] (B1) -- (M1) node[at start, below] {\tiny{$1$}};
\draw[->] (B2) -- (M2) node[at start, below] {\tiny{$m+l-1$}};
\draw[->] (M2) -- (M1) node[midway, above] {\tiny{$l-1$}};
\draw[->] (M2) -- (T2) node[at end, above] {\tiny{$m$}};
\draw[->] (M1) -- (T1) node[at end, above] {\tiny{$l$}};
}}}}}

\newcommand{\dottedsquare}{
\coordinate (B1) at (-0.9,0);
\coordinate (B2) at (0.9,0);
\coordinate (C1) at (-0.9,1);
\coordinate (D1) at (-0.9,2);
\coordinate (C2) at (0.9,1);
\coordinate (D2) at (0.9,2);
\coordinate (T1) at (-0.9,3);
\coordinate (T2) at (0.9,3);
\coordinate (B) at (0, -1);
\draw[dotted] (B1) -- (T1);
\draw[dotted] (B2) -- (T2);
\draw[dotted] (C1) -- (C2);
\draw[dotted] (D1) -- (D2);
\draw[white, opacity=0] (0, -0.5) -- (0, 3.5);
}

\newcommand{\twovertupdown}{\ensuremath{\vcenter{\hbox{\tikz[scale=0.2]{
\dottedsquare
\draw[->] (B1) -- (T1);
\draw[->] (T2) -- (B2);
\node at (B) {$0$};
}}}}}

\newcommand{\twovertdowndown}{\ensuremath{\vcenter{\hbox{\tikz[scale=0.2]{
\dottedsquare
\draw[<-] (B1) -- (T1);
\draw[->] (T2) -- (B2);
\node at (B) {$0$};
}}}}}

\newcommand{\twovertupup}{\ensuremath{\vcenter{\hbox{\tikz[scale=0.2]{
\dottedsquare
\draw[->] (B1) -- (T1);
\draw[<-] (T2) -- (B2);
\node at (B) {$0$};
}}}}}

\newcommand{\twovertupupNAME}[1]{\ensuremath{\vcenter{\hbox{\tikz[scale=0.3]{
\dottedsquare
\draw[->] (B1) -- (T1);
\draw[<-] (T2) -- (B2);
\node at (B) {$#1$};
}}}}}

\newcommand{\twovertdownup}{\ensuremath{\vcenter{\hbox{\tikz[scale=0.2]{
\dottedsquare
\draw[->] (B2) -- (T2);
\draw[->] (T1) -- (B1);
\node at (B) {$0$};
}}}}}

\newcommand{\onevertemptyup}{\ensuremath{\vcenter{\hbox{\tikz[scale=0.2]{
\dottedsquare
\draw[<-] (T2) -- (B2);
\node at (B) {$0$};
}}}}}

\newcommand{\onevertemptyupNAME}[1]{\ensuremath{\vcenter{\hbox{\tikz[scale=0.3]{
\dottedsquare
\draw[<-] (T2) -- (B2);
\node at (B) {$#1$};
}}}}}

\newcommand{\onevertemptydown}{\ensuremath{\vcenter{\hbox{\tikz[scale=0.2]{
\dottedsquare
\draw[->] (T2) -- (B2);
\node at (B) {$0$};
}}}}}

\newcommand{\onevertupempty}{\ensuremath{\vcenter{\hbox{\tikz[scale=0.2]{
\dottedsquare
\draw[<-] (T1) -- (B1);
\node at (B) {$0$};
}}}}}

\newcommand{\onevertupemptyNAME}[1]{\ensuremath{\vcenter{\hbox{\tikz[scale=0.3]{
\dottedsquare
\draw[<-] (T1) -- (B1);
\node at (B) {$#1$};
}}}}}

\newcommand{\onevertdownempty}{\ensuremath{\vcenter{\hbox{\tikz[scale=0.2]{
\dottedsquare
\draw[->] (T1) -- (B1);
\node at (B) {$0$};
}}}}}

\newcommand{\emptysquare}{\ensuremath{\vcenter{\hbox{\tikz[scale=0.2]{
\dottedsquare
\node at (B) {$0$};
}}}}}

\newcommand{\emptysquareNAME}[1]{\ensuremath{\vcenter{\hbox{\tikz[scale=0.3]{
\dottedsquare
\node at (B) {$#1$};
}}}}}

\newcommand{\twohorleftright}{\ensuremath{\vcenter{\hbox{\tikz[scale=0.2]{
\dottedsquare
\draw[->] (T2) -- (D2)-- (D1) -- (T1);
\draw[->] (B1) -- (C1)-- (C2) -- (B2);
\node at (B) {$-1$};
}}}}}

\newcommand{\twohorrightleft}{\ensuremath{\vcenter{\hbox{\tikz[scale=0.2]{
\dottedsquare
\draw[<-] (T2) -- (D2)-- (D1) -- (T1);
\draw[<-] (B1) -- (C1)-- (C2) -- (B2);
\node at (B) {$+1$};
}}}}}

\newcommand{\twohorleftleft}{\ensuremath{\vcenter{\hbox{\tikz[scale=0.2]{
\dottedsquare
\draw[->] (T2) -- (D2)-- (D1) -- (T1);
\draw[<-] (B1) -- (C1)-- (C2) -- (B2);
\node at (B) {$0$};
}}}}}

\newcommand{\twohorrightright}{\ensuremath{\vcenter{\hbox{\tikz[scale=0.2]{
\dottedsquare
\draw[<-] (T2) -- (D2)-- (D1) -- (T1);
\draw[->] (B1) -- (C1)-- (C2) -- (B2);
\node at (B) {$0$};
}}}}}

\newcommand{\onehordeeprightempty}{\ensuremath{\vcenter{\hbox{\tikz[scale=0.2]{
\dottedsquare
\draw[<-] (T2) -- (C2)-- (C1) -- (T1);
\node at (B) {$+\frac12$};
}}}}}

\newcommand{\onehordeepleftempty}{\ensuremath{\vcenter{\hbox{\tikz[scale=0.2]{
\dottedsquare
\draw[->] (T2) -- (C2)-- (C1) -- (T1);
\node at (B) {$-\frac12$};
}}}}}

\newcommand{\onehordeepemptyleft}{\ensuremath{\vcenter{\hbox{\tikz[scale=0.2]{
\dottedsquare
\draw[->] (B2) -- (D2)-- (D1) -- (B1);
\node at (B) {$+\frac12$};
}}}}}

\newcommand{\onehordeepemptyright}{\ensuremath{\vcenter{\hbox{\tikz[scale=0.2]{
\dottedsquare
\draw[<-] (B2) -- (D2)-- (D1) -- (B1);
\node at (B) {$-\frac12$};
}}}}}


\newcommand{\onehorrightempty}{\ensuremath{\vcenter{\hbox{\tikz[scale=0.2]{
\dottedsquare
\draw[<-] (T2) -- (D2)-- (D1) -- (T1);
\node at (B) {$+\frac12$};
}}}}}

\newcommand{\onehorleftempty}{\ensuremath{\vcenter{\hbox{\tikz[scale=0.2]{
\dottedsquare
\draw[->] (T2) -- (D2)-- (D1) -- (T1);
\node at (B) {$-\frac12$};
}}}}}

\newcommand{\onehoremptyleft}{\ensuremath{\vcenter{\hbox{\tikz[scale=0.2]{
\dottedsquare
\draw[->] (B2) -- (C2)-- (C1) -- (B1);
\node at (B) {$+\frac12$};
}}}}}

\newcommand{\onehoremptyright}{\ensuremath{\vcenter{\hbox{\tikz[scale=0.2]{
\dottedsquare
\draw[<-] (B2) -- (C2)-- (C1) -- (B1);
\node at (B) {$-\frac12$};
}}}}}

\newcommand{\bottomlefttoprightlow}{\ensuremath{\vcenter{\hbox{\tikz[scale=0.2]{
\dottedsquare
\draw[->] (B1) -- (C1)-- (C2) -- (T2);
\node at (B) {$0$};
}}}}}

\newcommand{\bottomlefttoprightlowNAME}[1]{\ensuremath{\vcenter{\hbox{\tikz[scale=0.3]{
\dottedsquare
\draw[->] (B1) -- (C1)-- (C2) -- (T2);
\node at (B) {$#1$};
}}}}}

\newcommand{\bottomlefttoprighthigh}{\ensuremath{\vcenter{\hbox{\tikz[scale=0.2]{
\dottedsquare
\draw[->] (B1) -- (D1)-- (D2) -- (T2);
\node at (B) {$0$};
}}}}}

\newcommand{\bottomlefttoprighthighNAME}[1]{\ensuremath{\vcenter{\hbox{\tikz[scale=0.3]{
\dottedsquare
\draw[->] (B1) -- (D1)-- (D2) -- (T2);
\node at (B) {$#1$};
}}}}}

\newcommand{\bottomrighttopleftlow}{\ensuremath{\vcenter{\hbox{\tikz[scale=0.2]{
\dottedsquare
\draw[->] (B2) -- (C2)-- (C1) -- (T1);
\node at (B) {$0$};
}}}}}

\newcommand{\bottomrighttopleftlowNAME}[1]{\ensuremath{\vcenter{\hbox{\tikz[scale=0.3]{
\dottedsquare
\draw[->] (B2) -- (C2)-- (C1) -- (T1);
\node at (B) {$#1$};
}}}}}

\newcommand{\bottomrighttoplefthigh}{\ensuremath{\vcenter{\hbox{\tikz[scale=0.2]{
\dottedsquare
\draw[->] (B2) -- (D2)-- (D1) -- (T1);
\node at (B) {$0$};
}}}}}

\newcommand{\bottomrighttoplefthighNAME}[1]{\ensuremath{\vcenter{\hbox{\tikz[scale=0.3]{
\dottedsquare
\draw[->] (B2) -- (D2)-- (D1) -- (T1);
\node at (B) {$#1$};
}}}}}

\newcommand{\toprightbottomleftlow}{\ensuremath{\vcenter{\hbox{\tikz[scale=0.2]{
\dottedsquare
\draw[<-] (B1) -- (C1)-- (C2) -- (T2);
\node at (B) {$0$};
}}}}}

\newcommand{\toprightbottomlefthigh}{\ensuremath{\vcenter{\hbox{\tikz[scale=0.2]{
\dottedsquare
\draw[<-] (B1) -- (D1)-- (D2) -- (T2);
\node at (B) {$0$};
}}}}}

\newcommand{\topleftbottomrightlow}{\ensuremath{\vcenter{\hbox{\tikz[scale=0.2]{
\dottedsquare
\draw[<-] (B2) -- (C2)-- (C1) -- (T1);
\node at (B) {$0$};
}}}}}

\newcommand{\topleftbottomrighthigh}{\ensuremath{\vcenter{\hbox{\tikz[scale=0.2]{
\dottedsquare
\draw[<-] (B2) -- (D2)-- (D1) -- (T1);
\node at (B) {$0$};
}}}}}

\newcommand{\zigzagleftdown}{\ensuremath{\vcenter{\hbox{\tikz[scale=0.2]{
\dottedsquare
\draw[->] (T1) -- (D1)-- (D2) -- (C2) -- (C1) -- (B1);
\node at (B) {$0$};
}}}}}

\newcommand{\zigzagleftup}{\ensuremath{\vcenter{\hbox{\tikz[scale=0.2]{
\dottedsquare
\draw[<-] (T1) -- (D1)-- (D2) -- (C2) -- (C1) -- (B1);
\node at (B) {$0$};
}}}}}

\newcommand{\zigzagleftupNAME}[1]{\ensuremath{\vcenter{\hbox{\tikz[scale=0.3]{
\dottedsquare
\draw[<-] (T1) -- (D1)-- (D2) -- (C2) -- (C1) -- (B1);
\node at (B) {$#1$};
}}}}}

\newcommand{\zigzagrightdown}{\ensuremath{\vcenter{\hbox{\tikz[scale=0.2]{
\dottedsquare
\draw[->] (T2) -- (D2)-- (D1) -- (C1) -- (C2) -- (B2);
\node at (B) {$0$};
}}}}}

\newcommand{\zigzagrightup}{\ensuremath{\vcenter{\hbox{\tikz[scale=0.2]{
\dottedsquare
\draw[<-] (T2) -- (D2)-- (D1) -- (C1) -- (C2) -- (B2);
\node at (B) {$0$};
}}}}}

\newcommand{\zigzagrightupNAME}[1]{\ensuremath{\vcenter{\hbox{\tikz[scale=0.3]{
\dottedsquare
\draw[<-] (T2) -- (D2)-- (D1) -- (C1) -- (C2) -- (B2);
\node at (B) {$#1$};
}}}}}

\newcommand{\squaremiddlepos}{\ensuremath{\vcenter{\hbox{\tikz[scale=0.2]{
\dottedsquare
\draw[->] (D2)-- (D1) -- (C1);
\draw[->] (C1)-- (C2) -- (D2);
\node at (B) {$+1$};
}}}}}

\newcommand{\squaremiddleneg}{\ensuremath{\vcenter{\hbox{\tikz[scale=0.2]{
\dottedsquare
\draw[<-] (D2)-- (D1) -- (C1);
\draw[<-] (C1)-- (C2) -- (D2);
\node at (B) {$-1$};
}}}}}

\newcommand{\bigX}{\ensuremath{\vcenter{\hbox{\tikz[scale=0.18]{
\draw[very thick] (-1,1) -- (1, -1);
\draw[very thick] (1,1) -- (-1, -1);
}}}}}

\newcommand{\dotteddoubleY}{
\coordinate (B1) at (-0.9,0);
\coordinate (T1) at (-0.9,3);
\coordinate (C) at (0,1);
\coordinate (D) at (0,2);
\coordinate (B2) at (0.9,0);
\coordinate (T2) at (0.9,3);
\coordinate (B) at (0,-1);
\draw[white] (0, -0.5) -- (0, 3.5);
\draw[dotted] (B1) -- (C);
\draw[dotted] (C) -- (B2);
\draw[dotted] (D) -- (C);
\draw[dotted] (T2) -- (D);
\draw[dotted] (D) -- (T1);
}

\newcommand{\Yonevertupempty}{\ensuremath{\vcenter{\hbox{\tikz[scale=0.2]{
\dotteddoubleY
\draw[->] (B1)-- (C) -- (D) -- (T1);
\node at (B) {$0$};
}}}}}

\newcommand{\Yonevertdownempty}{\ensuremath{\vcenter{\hbox{\tikz[scale=0.2]{
\dotteddoubleY
\draw[<-] (B1)-- (C) -- (D) -- (T1);
\node at (B) {$0$};
}}}}}

\newcommand{\Ytwohorleftright}{\ensuremath{\vcenter{\hbox{\tikz[scale=0.2]{
\dotteddoubleY
\draw[->] (T2) -- (D) -- (T1);
\draw[->] (B1) -- (C) -- (B2);
\node at (B) {$-1$};
}}}}}

\newcommand{\Ytwohorrightleft}{\ensuremath{\vcenter{\hbox{\tikz[scale=0.2]{
\dotteddoubleY
\draw[<-] (T2) -- (D) -- (T1);
\draw[<-] (B1) -- (C) -- (B2);
\node at (B) {$+1$};
}}}}}

\newcommand{\Ytwohorleftleft}{\ensuremath{\vcenter{\hbox{\tikz[scale=0.2]{
\dotteddoubleY
\draw[->] (T2) -- (D) -- (T1);
\draw[<-] (B1) -- (C) -- (B2);
\node at (B) {$0$};
}}}}}

\newcommand{\Ytwohorrightright}{\ensuremath{\vcenter{\hbox{\tikz[scale=0.2]{
\dotteddoubleY
\draw[<-] (T2) -- (D) -- (T1);
\draw[->] (B1) -- (C) -- (B2);
\node at (B) {$0$};
}}}}}

\newcommand{\Yonehoremptyright}{\ensuremath{\vcenter{\hbox{\tikz[scale=0.2]{
\dotteddoubleY
\draw[->] (B1) -- (C) -- (B2);
\node at (B) {$-\frac12$};
}}}}}

\newcommand{\Yonehoremptyleft}{\ensuremath{\vcenter{\hbox{\tikz[scale=0.2]{
\dotteddoubleY
\draw[<-] (B1) -- (C) -- (B2);
\node at (B) {$+\frac12$};
}}}}}

\newcommand{\Yonehorrightempty}{\ensuremath{\vcenter{\hbox{\tikz[scale=0.2]{
\dotteddoubleY
\draw[->] (T1) -- (D) -- (T2);
\node at (B) {$+\frac12$};
}}}}}

\newcommand{\Yonehorleftempty}{\ensuremath{\vcenter{\hbox{\tikz[scale=0.2]{
\dotteddoubleY
\draw[->] (T2) -- (D) -- (T1);
\node at (B) {$-\frac12$};
}}}}}

\newcommand{\Yonevertemptyup}{\ensuremath{\vcenter{\hbox{\tikz[scale=0.2]{
\dotteddoubleY
\draw[->] (B2)-- (C) -- (D) -- (T2);
\node at (B) {$0$};
}}}}}

\newcommand{\Yonevertemptydown}{\ensuremath{\vcenter{\hbox{\tikz[scale=0.2]{
\dotteddoubleY
\draw[<-] (B2)-- (C) -- (D) -- (T2);
\node at (B) {$0$};
}}}}}

\newcommand{\Yempty}{\ensuremath{\vcenter{\hbox{\tikz[scale=0.2]{
\dotteddoubleY
\node at (B) {$0$};
}}}}}

\newcommand{\Ybottomrighttopleft}{\ensuremath{\vcenter{\hbox{\tikz[scale=0.2]{
\dotteddoubleY
\draw[->] (B2)-- (C) -- (D) -- (T1);
\node at (B) {$0$};
}}}}}

\newcommand{\Ybottomlefttopright}{\ensuremath{\vcenter{\hbox{\tikz[scale=0.2]{
\dotteddoubleY
\draw[->] (B1)-- (C) -- (D) -- (T2);
\node at (B) {$0$};
}}}}}

\newcommand{\Ytopleftbottomright}{\ensuremath{\vcenter{\hbox{\tikz[scale=0.2]{
\dotteddoubleY
\draw[<-] (B2)-- (C) -- (D) -- (T1);
\node at (B) {$0$};
}}}}}

\newcommand{\Ytoprightbottomleft}{\ensuremath{\vcenter{\hbox{\tikz[scale=0.2]{
\dotteddoubleY
\draw[<-] (B1)-- (C) -- (D) -- (T2);
\node at (B) {$0$};
}}}}}

\newcommand{\dottedII}{
\coordinate (B1) at (-1,0);
\coordinate (T1) at (-1,3);
\coordinate (B2) at (1,0);
\coordinate (T2) at (1,3);
\coordinate (B) at (0,-1);
\draw[white] (0, -0.5) -- (0, 3.5);
\draw[dotted] (B1) -- (T1);
\draw[dotted] (B2) -- (T2);
}

\newcommand{\IItwovertupup}{\ensuremath{\vcenter{\hbox{\tikz[scale=0.2]{
\dottedII
\draw[->] (B2)-- (T2);
\draw[->] (B1)-- (T1);
\node at (B) {$0$};
}}}}}

\newcommand{\IItwovertdownup}{\ensuremath{\vcenter{\hbox{\tikz[scale=0.2]{
\dottedII
\draw[<-] (B1)-- (T1);
\draw[->] (B2)-- (T2);
\node at (B) {$0$};
}}}}}

\newcommand{\IItwovertupdown}{\ensuremath{\vcenter{\hbox{\tikz[scale=0.2]{
\dottedII
\draw[->] (B1)-- (T1);
\draw[<-] (B2)-- (T2);
\node at (B) {$0$};
}}}}}

\newcommand{\IItwovertdowndown}{\ensuremath{\vcenter{\hbox{\tikz[scale=0.2]{
\dottedII
\draw[<-] (B1)-- (T1);
\draw[<-] (B2)-- (T2);
\node at (B) {$0$};
}}}}}

\newcommand{\IIonevertupempty}{\ensuremath{\vcenter{\hbox{\tikz[scale=0.2]{
\dottedII
\draw[->] (B1)-- (T1);
\node at (B) {$0$};
}}}}}

\newcommand{\IIonevertdownempty}{\ensuremath{\vcenter{\hbox{\tikz[scale=0.2]{
\dottedII
\draw[<-] (B1)-- (T1);
\node at (B) {$0$};
}}}}}

\newcommand{\IIonevertemptydown}{\ensuremath{\vcenter{\hbox{\tikz[scale=0.2]{
\dottedII
\draw[<-] (B2)-- (T2);
\node at (B) {$0$};
}}}}}

\newcommand{\IIonevertemptyup}{\ensuremath{\vcenter{\hbox{\tikz[scale=0.2]{
\dottedII
\draw[->] (B2)-- (T2);
\node at (B) {$0$};
}}}}}

\newcommand{\IIempty}{\ensuremath{\vcenter{\hbox{\tikz[scale=0.2]{
\dottedII
\node at (B) {$0$};
}}}}}

\newcommand{\dottedH}{
\coordinate (B1) at (-1,0);
\coordinate (T1) at (-1,3);
\coordinate (B2) at (1,0);
\coordinate (T2) at (1,3);
\coordinate (M1) at (-1,1.5);
\coordinate (M2) at (1,1.5);
\coordinate (B) at (0,-1);
\draw[white] (0, -0.5) -- (0, 3.5);
\draw[dotted] (B1) -- (T1);
\draw[dotted] (B2) -- (T2);
\draw[dotted] (M2) -- (M1);
}

\newcommand{\Hempty}{\ensuremath{\vcenter{\hbox{\tikz[scale=0.2]{
\dottedH
\node at (B) {$0$};
}}}}}

\newcommand{\Honevertemptyup}{\ensuremath{\vcenter{\hbox{\tikz[scale=0.2]{
\dottedH
\draw[->] (B2) -- (T2);
\node at (B) {$0$};
}}}}}

\newcommand{\Honevertemptydown}{\ensuremath{\vcenter{\hbox{\tikz[scale=0.2]{
\dottedH
\draw[<-] (B2) -- (T2);
\node at (B) {$0$};
}}}}}

\newcommand{\Honevertupempty}{\ensuremath{\vcenter{\hbox{\tikz[scale=0.2]{
\dottedH
\draw[->] (B1) -- (T1);
\node at (B) {$0$};
}}}}}

\newcommand{\Honevertdownempty}{\ensuremath{\vcenter{\hbox{\tikz[scale=0.2]{
\dottedH
\draw[<-] (B1) -- (T1);
\node at (B) {$0$};
}}}}}

\newcommand{\Htwovertupup}{\ensuremath{\vcenter{\hbox{\tikz[scale=0.2]{
\dottedH
\draw[->] (B1) -- (T1);
\draw[->] (B2) -- (T2);
\node at (B) {$0$};
}}}}}

\newcommand{\Htwovertupdown}{\ensuremath{\vcenter{\hbox{\tikz[scale=0.2]{
\dottedH
\draw[->] (B1) -- (T1);
\draw[<-] (B2) -- (T2);
\node at (B) {$0$};
}}}}}

\newcommand{\Htwovertdownup}{\ensuremath{\vcenter{\hbox{\tikz[scale=0.2]{
\dottedH
\draw[<-] (B1) -- (T1);
\draw[->] (B2) -- (T2);
\node at (B) {$0$};
}}}}}

\newcommand{\Htwovertdowndown}{\ensuremath{\vcenter{\hbox{\tikz[scale=0.2]{
\dottedH
\draw[<-] (B1) -- (T1);
\draw[<-] (B2) -- (T2);
\node at (B) {$0$};
}}}}}

\newcommand{\Hbottomlefttopright}{\ensuremath{\vcenter{\hbox{\tikz[scale=0.2]{
\dottedH
\draw[->] (B1) -- (M1)-- (M2) -- (T2);
\node at (B) {$0$};
}}}}}

\newcommand{\Htoprightbottomleft}{\ensuremath{\vcenter{\hbox{\tikz[scale=0.2]{
\dottedH
\draw[<-] (B1) -- (M1)-- (M2) -- (T2);
\node at (B) {$0$};
}}}}}

\newcommand{\Hbottomrighttopleft}{\ensuremath{\vcenter{\hbox{\tikz[scale=0.2]{
\dottedH
\draw[->] (B2) -- (M2)-- (M1) -- (T1);
\node at (B) {$0$};
}}}}}

\newcommand{\Htopleftbottomright}{\ensuremath{\vcenter{\hbox{\tikz[scale=0.2]{
\dottedH
\draw[<-] (B2) -- (M2)-- (M1) -- (T1);
\node at (B) {$0$};
}}}}}

\newcommand{\Honehorleftempty}{\ensuremath{\vcenter{\hbox{\tikz[scale=0.2]{
\dottedH
\draw[->] (T2) -- (M2)-- (M1) -- (T1);
\node at (B) {$-\frac12$};
}}}}}

\newcommand{\Honehorrightempty}{\ensuremath{\vcenter{\hbox{\tikz[scale=0.2]{
\dottedH
\draw[<-] (T2) -- (M2)-- (M1) -- (T1);
\node at (B) {$+\frac12$};
}}}}}

\newcommand{\Honehoremptyleft}{\ensuremath{\vcenter{\hbox{\tikz[scale=0.2]{
\dottedH
\draw[->] (B2) -- (M2)-- (M1) -- (B1);
\node at (B) {$+\frac12$};
}}}}}

\newcommand{\Honehoremptyright}{\ensuremath{\vcenter{\hbox{\tikz[scale=0.2]{
\dottedH
\draw[<-] (B2) -- (M2)-- (M1) -- (B1);
\node at (B) {$-\frac12$};
}}}}}

\begin{abstract}
This paper is concerned with evaluation of $\sll_N$-webs from a graph-theoretical point of view: We give an interpretation of the evaluation of $\sll_N$ webs in terms of colorings. This is very close from the approach of Cautis, Kamnitzer and Morrison, but we provide a non-local and algebra-free definition of the degree associated with a coloring. In particular we do not use skew Howe duality. As a counter-part we are only concerned with closed webs. We prove that this new evaluation coincides with the classical evaluation of MOY graphs by checking some skein relations. As a consequence, we prove a formula which relates the $\sll_N$ and $\sll_{N-1}$-evaluations of MOY graphs.
\end{abstract}
\maketitle
\setcounter{tocdepth}{1}
\tableofcontents
\allowdisplaybreaks
\section{Introduction}
\label{sec:introduction}

MOY graphs and MOY graph evaluation have been introduced by Murakami Ohtsuki and Yamada~\cite{MR1659228} to provide a combinatorial and computational approach to the $\sll_N$-invariant of links~\cite{MR1036112}. The edges of these graphs are meant to represent some wedge powers of the fundamental representation of the Hopf algebra $U_q(\sll_N)$ and the vertices correspond to some intertwiners. A MOY graph can therefore be interpreted as an endomorphism of $\CC(q)$ the trivial $U_q(\sll_N)$-module. The evaluation of a MOY graph is the image of 1 under this endomorphism. 

Murakami, Ohtsuki and Yamada gave a combinatorial way to evaluate MOY graphs using colorings and state sums and gave some skein relations satisfied by the evaluation.  Kim \cite{2003Kim} and Morrison \cite{2007arXiv0704.1503M} conjectured that (a Karoubi-completion of) the category of MOY graphs is equivalent to that of finite dimensional $U_q(\sll_N)$-modules. This has been proved by Cautis, Kamnitzer and Morrison \cite{MR3263166}. 

The combinatorial evaluation of MOY graphs has been used to study the categorification of the $\sll_N$-invariant (see for example \cite{MR2391017,MR2421131,pre06302580,MR2491657,MR3190356,queffelec2014mathfrak, MR3263166}).  

In this paper, we give an alternative definition of the evaluation of MOY graphs. This new definition is very close from that of \cite{MR3263166}: just like Cautis, Kamnitzer and Morrison, we count all possible colorings of a MOY graphs taking into account a certain degree. As we only work with closed MOY graphs, many simplifications occur, giving us a global definition of the degree (this implies in particular that we do not need their \emph{tags}). 
In other word, this paper can be seen as a combinatorial rewriting of part of \cite{MR3263166} in the case of closed MOY graphs.




From this new evaluation, one can deduce a skein relation which relates $\sll_N$ and $\sll_{N-1}$-evaluations of MOY graphs (this formula can be as well derived from the diagrammatic description of the Gel'fand-Tsetlin functor in the PhD thesis of Morrison \cite[Chapter 4]{2007arXiv0704.1503M}).

We hope that our fully-combinatorial approach to the evaluation of closed MOY graphs can be applied to foam-theoretic categorifications of the $\sll_N$-invariant. In particular, we think that it could helpful to get rid of the ``ladder'' formalism used, see for example \cite{queffelec2014mathfrak}.  

\subsection*{Organization of the paper}
\label{sec:organisation-paper}

In section~\ref{sec:eval-sll_n-webs}, we define our evaluation $\kupcs{\cdot}$ of $\sll_N$-webs and we state in Theorem~\ref{thm:evalcol} that our evaluation agrees with the one defined in \cite{MR1659228}. We explain that in order to prove the theorem, it is enough to show that $\kupcs{\cdot}$ satisfies some skein relations.

In section~\ref{sec:part-q-ident}, we introduce degrees of partitions of ordered sets and show a relatively technical lemma about this notion which is the key point of the proof of Theorem~\ref{thm:evalcol}.

In section~\ref{sec:check-skein-relat}, we prove that $\kupcs{\cdot}$ satisfies the skein relations. Finally in section~\ref{sec:new-kind-skein}, we state and prove Proposition~\ref{prop:newskein} which related the $\sll_N$ evaluation and the $\sll_{N-1}$ evaluation of MOY graphs.

 \subsection*{Acknowledgement}
 \label{sec:acknoledgement}
 The author would like to thank Christian Blanchet, Ben Elias and Jan Priel for showing interest in this work and for their helpful comments.

\section{Evaluation of $\sll_N$-webs}
\label{sec:eval-sll_n-webs}

\newcommand{\cN}{\ensuremath{\llbracket 1, N\rrbracket}}
\newcommand{\PP}{\ensuremath{\mathfrak{P}}}
\begin{dfn}
  \label{dfn:MOYGraph}
Let $N$ be a positive integer. An \emph{$\sll_N$-web} or simply \emph{web} is an oriented, trivalent, plane graph $\Gamma = (V,E)$ with possibly some vertex-less loops, whose edges are labeled with elements of $\ZZ$\footnote{Actually, if the label of an edge is not in $\llbracket 0, N\rrbracket := \{0,1, \dots, N\}$, the web will be pretty uninteresting. However, for technical reasons, it is convenient to allow labelling by all relative integers.}, such that for every vertex $v$ of $\Gamma$,  we have:
\[
\sum_{v \stackrel{e}{\longrightarrow} \bullet} \lambda(e) =  \sum_{v \stackrel{e}{\longleftarrow} \bullet} \lambda(e) \mod N,
\]
where $\lambda: E \to \cN$ is the labelling of the edges. 
Two $\sll_N$-webs $\Gamma_1$ and $\Gamma_2$ are considered to be equivalent if one can obtain one from the other by reversing the orientations of some edges $(e)_{e\in E'}$ and replacing their labels by $N- \lambda(e)$.
\end{dfn}

\begin{dfn} \label{dfn:MOYGraph2}
  A \emph{MOY graph} is a trivalent oriented plane graph $\Gamma$ with possibly some vertex-less loops, whose edges are labeled with elements of $\NN$, such that for every vertex the labels and the orientations look locally like:
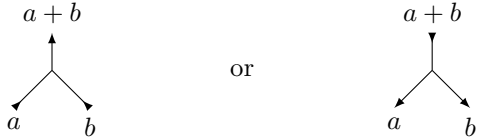
\[
\tikz{
\draw[->] (0,0) -- (0,0.5) node [at end, above] {$a+b$};  
\draw[>-] (-0.5, -0.5) -- (0,0) node [at start, below] {$a$};  
\draw[>-] (+0.5, -0.5) -- (0,0) node [at start, below] {$b$};  
\node at (2.5,0) {or};
\begin{scope}[xshift = 5cm]
  \draw[-<] (0,0) -- (0,0.5) node [at end, above] {$a+b$};  
\draw[<-] (-0.5, -0.5) -- (0,0) node [at start, below] {$a$};  
\draw[<-] (+0.5, -0.5) -- (0,0) node [at start, below] {$b$};  
\end{scope}
}
\]
If $\Gamma$ is a MOY graph, the \emph{writhe of $\Gamma$}, denoted by $w(\Gamma)$, is the algebraic number of circles obtained, when one replaces every edge with label $i$ of $\Gamma$ by $i$ parallel copies (see Figure~\ref{fig:writhe}).
\begin{figure}[ht!]
  \centering
  \begin{tikzpicture}[yscale = 0.4]
    \begin{scope}
  \coordinate (A) at (0,0);
  \coordinate (B) at (0,2);
  \coordinate (C) at (0,4);
  \coordinate (D) at (0,6);
  \node at (0, -1.5) {$\Gamma$};
  \draw[->] (A) -- (B) node[midway, left] {$4$};
  \draw[->] (C) -- (D) node[midway, left] {$4$};
  \draw[->] (B) .. controls +(0.5,0.5) and +(0.5,-0.5) .. (C) node[midway, left] {$3$};
  \draw[->] (B) .. controls +(-0.5,0.5) and +(-0.5,-0.5) .. (C) node[midway, left] {$1$};
  \draw[->] (D) .. controls +(1.5,1.5) and +(1.5,-1.5) .. (A) node[midway, left] {$2$};
  \draw[->] (D) .. controls +(-1.5,1.5) and +(-1.5,-1.5) .. (A) node[midway, left] {$2$};
\end{scope}

\begin{scope}[xshift = 5cm]
  \coordinate (A1) at (-0.15,0);
  \coordinate (A2) at (-0.05,0);
  \coordinate (A4) at (+0.15,0);
  \coordinate (A3) at (+0.05,0);
  \coordinate (B1) at (-0.15,2);
  \coordinate (B2) at (-0.05,2);
  \coordinate (B4) at (+0.15,2);
  \coordinate (B3) at (+0.05,2);
  \coordinate (C1) at (-0.15,4);
  \coordinate (C2) at (-0.05,4);
  \coordinate (C4) at (+0.15,4);
  \coordinate (C3) at (+0.05,4);
  \coordinate (D1) at (-0.15,6);
  \coordinate (D2) at (-0.05,6);
  \coordinate (D4) at (+0.15,6);
  \coordinate (D3) at (+0.05,6);
  \draw[->] (A1) -- (B1);
  \draw[->] (A2) -- (B2);
  \draw[->] (A3) -- (B3);
  \draw[->] (A4) -- (B4);
  \draw[->] (C1) -- (D1);
  \draw[->] (C2) -- (D2);
  \draw[->] (C3) -- (D3);
  \draw[->] (C4) -- (D4);
  \draw[->] (B1) .. controls +(-0.5,0.5) and +(-0.5,-0.5) .. (C1);
  \draw[->] (B2) .. controls +(-0.5,0.5) and +(-0.5,-0.5) .. (C2);
  \draw[->] (B3) .. controls +(-0.5,0.5) and +(-0.5,-0.5) .. (C3);
  \draw[->] (B4) .. controls +(0.5,0.5) and +(0.5,-0.5) .. (C4); 
   \draw[->] (D3) .. controls +(1.65,1.9) and +(1.65,-1.9) .. (A3);
   \draw[->] (D4) .. controls +(1.4,1.3) and +(1.4,-1.3) .. (A4);
   \draw[->] (D1) .. controls +(-1.4,1.3) and +(-1.4,-1.3) .. (A1);
   \draw[->] (D2) .. controls +(-1.65,1.9) and +(-1.65,-1.9) .. (A2);
\end{scope}
  \end{tikzpicture}
  \caption{Computation of the writhe of a MOY graph: here we have $w(\Gamma)=2-2 =0$}
  \label{fig:writhe}
\end{figure}
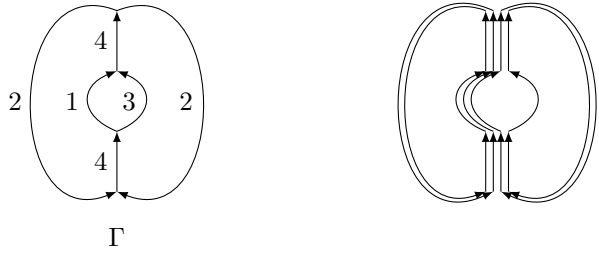
\end{dfn}

\begin{dfn}
  \label{dfn:colofwebs}
  A \emph{coloring} of a web $\Gamma = (V,E)$ is a function $c$  from $E$ to $\PP(\cN)$  such that:
  \begin{enumerate}
  \item[(E)] for every edge $e$, $\#c(e)= \lambda (e)$,
  \item[(V)] for every vertex $v$, the multiset:
\[
\bigcup_{v \stackrel{e}{\longrightarrow} \bullet} c(e) \cup   \bigcup_{v \stackrel{e}{\longleftarrow} \bullet}  \overline{c(e)}
 \]  
is a multiple of $\cN$, where $ \overline{c(e)}$ is the complement of $c(e)$ in $\cN$.
  \end{enumerate}
If two webs $\Gamma_1$ and $\Gamma_2$ are equivalent and $c_1$ is a coloring of $\Gamma_1$, there is a canonical coloring of $\Gamma_2$ obtained by replacing $c_1(e)$ by its complement in $\cN$ for every edge $e$ whose orientation has been reversed. 
\end{dfn}

\begin{rmk}\label{rmk:flow}
  \begin{enumerate}
  \item If some labels of an $\sll_N$-web $\Gamma$ are not in $\llbracket 0,N\rrbracket$, then the web $\Gamma$ admit no coloring.
\item If $\Gamma$ is a MOY graph (and hence can be thought of as an $\sll_N$-web for all $N$), the condition (V) of a coloring is equivalent to saying that around each vertices, the colors of the two edges with the smallest labels form a partition of the color of the edge with the greatest label. Therefore, for each element $i$ in $\cN$, the flow of $i$ is preserved around each vertex and one can see the coloring of $\Gamma$ as a collection of connected cycles colored by element of $\cN$ such that:
  \begin{itemize}
  \item any two cycles with the same colors are disjoint,
  \item an edge $e$ belongs to exactly $\lambda(e)$ cycles.
  \end{itemize}
  \end{enumerate}
\end{rmk}

\begin{dfn}
    \label{dfn:state}
    A \emph{bicolor} $b$ is a subset of $\cN$ with exactly two elements. The greatest color (for the natural order on $\cN$) is denoted by $b^+$, the other one by $b^-$. 
    If $(\Gamma,c)$ is a colored web and $b$ is a bicolor, the \emph{state} $(\Gamma,c)_b$ is the collection of oriented circles obtained from $\Gamma$ by erasing every edge $e$ such that the cardinal of the intersection of $c(e)$ and $b$ is different from $1$ and reversing the orientation of every edge $e$ such that $\lambda(e)\cap b = \{b^-\}$. The \emph{degree} $d(\Gamma,c)_b$ of a state $(\Gamma,c)_b$ is equal to the algebraic number of circles in $(\Gamma,c)_b$. The degree $d(\Gamma,c)$ of a colored web $(\Gamma,c)$ is the sum of the degree of all the possible states.  
  \end{dfn}
\begin{thm}
  \label{thm:evalcol}
  Let $\Gamma$ be a web, the evaluation $\kups{\Gamma}$ of $\Gamma$ given in \cite{MR1659228} is equal to:
\[
\kupcs{\Gamma} := \sum_{c \textrm{ coloring of } \Gamma} q^{d(\Gamma,c)}.
\]
\end{thm}
  \begin{rmk}\label{rmk:eqwebcoloring}
    \begin{enumerate}
    \item It is worthwhile to note that, if $\Gamma_1$ and $\Gamma_2$ are two equivalent webs, $c_1$ a coloring of $\Gamma_1$ and $c_2$ the corresponding coloring of $\Gamma_2$, then for every bicolor $b$, the states $(\Gamma_1, c_1)_b$ and $d(\Gamma_2, c_2)_b$ are equal. It follows that $d(\Gamma_1, c_1) = d(\Gamma_2, c_2)$ and $\kupcs{\Gamma_1} = \kupcs{\Gamma_2}$.
    \item Let $\Gamma$ be a web. The graph $\Gamma'$ obtained by removing all edges\footnote{Formally, one should as well remove the vertices such that all adjacent edges are labelled by $0$ or $N$.} labelled by $0$ or $N$ is a web and we have $\kupcs\Gamma = \kupcs{\Gamma'}$. 
    \item The theorem can be seen as a generalization of \cite[Theorem 1.11]{2013arXiv1312.0361R}, one can therefore wonder if the other result of \cite{2013arXiv1312.0361R} remains true in our context, namely: are all colorings of a web $\Gamma$ Kempe equivalent?  
    \end{enumerate}
  \end{rmk}

The evaluation of MOY graphs $\kups{\cdot}$ is multiplicative with respect to the disjoint union  and it satisfies the following skein relations:

\begin{align}
  \kups{\vcenter{\hbox{\tikz[scale= 0.5]{\draw[->] (0,0) arc(0:360:1cm) node[right] {\small{$\!k\!$}};}}}}=
\begin{pmatrix}
  N \\ k
\end{pmatrix}_{\!\! q}.\label{eq:relcircle}
\end{align}
\begin{align}
   \kup{\stgamma} = \kup{\stgammaprime} \label{eq:relass}\\
 \kup{\digona} = \arraycolsep=2.5pt
  \begin{pmatrix}
    m+n \\ m
  \end{pmatrix}_{\!\!q}
\kup{\verta} \label{eq:relbin1}
\end{align}
\begin{align}
\arraycolsep=2.5pt
\kup{\digonb} = 
  \begin{pmatrix}
    N-m \\ n
  \end{pmatrix}_{\!\!q}
\kup{\vertb}  \label{eq:relbin2}
\end{align}
\begin{align}
 \kup{\squarea} = \kup{\twoverta} + [N-m-1]_q\kup{\doubleYa} \label{eq:relsquare1}
\end{align}
\begin{align}
\kup{\squareb}=\!
  \begin{pmatrix}
    m-1 \\ n
  \end{pmatrix}_{\!\!q}\!
\kup{\bigHb}  +
\!\begin{pmatrix}
  m-1 \\n-1
\end{pmatrix}_{\!\!q}\!
\kup{\doubleYb} \label{eq:relsquare2}
\end{align}
\begin{align}
  \kup{\squarec}= \sum_{j=0}^\infty\begin{pmatrix}l \\ k-j \end{pmatrix}_{\!\!q} \kup{\squared}\label{eq:relsquare3}
\end{align}

In the previous formulas, $q$ is a formal variable, $[k]_q:= \frac{q^{+k}- q^{-k}}{q^{+1}- q^{-1}}$ and 
\[
  \begin{pmatrix}
    k \\ l
  \end{pmatrix}_{\!\!q} :=
  \begin{cases}
    0 & \textrm{ if $k<0$, $l<0$ or $l>k$,} \\
    \frac{[k]_q!}{[l]_q! [l-k]_q!} &\textrm{else,}
  \end{cases}
\quad \textrm{where } [i]!_q = \prod_{j=1}^i [j]_q.
\]
In the sequel, we will often write $\qbinb{l}{k}{l-k}$ for $\begin{pmatrix}  k \\ l\end{pmatrix}_{\!\!q}$ to emphasize the symmetry between $k$ and $l-k$ in the definition.

Wu \cite{pre06302580} proved that these relations characterize $\kups{\cdot}$. Hence to prove Theorem~\ref{thm:evalcol}, it is enough to prove that $\kupcs{\cdot}$ satisfies the same local relations.

In order to check the skein relation, it is convenient to have a local definition for $\kupcs{\cdot}$. For this purpose, we need some definitions:

\begin{dfn}
  \label{dfn:openweb}
  An \emph{open web} is a generic intersection of a web $\Gamma$ with $\RR\times [0,1]$ where $I$ is a non-empty interval. By "generic" we mean that:
  \begin{itemize}
  \item $\RR \times \{0,1\}$ does not intersect any vertex of $\Gamma$.
  \item the intersection of $\RR \times I$ and the edges of $\Gamma$ are transverse.
  \end{itemize}
  
The \emph{boundary} $\partial \Gamma$ of $\Gamma$ consists of intersection of $\Gamma$ with $\RR \times \partial[0,1]$ together with the orientations and the labels induced by $\Gamma$. A \emph{coloring} of an open web is defined exactly as for webs. A coloring $c$ of an open web induces a coloring $c_{\partial}$ of its boundary. 
\end{dfn}

It is worthwhile to remark that one may stack\footnote{Or compose, if one think of open webs as morphism in a suitable category.} two open webs (with some compatibility conditions of part of the boundaries) to obtain a new open web. If the colorings agree on the common boundaries, one can even stack colored webs.

\begin{dfn}
  \label{dfn:stateopen}
  If $(\Gamma,c)$ is a colored open web and $b$ is a bicolor, the \emph{state} $(\Gamma,c)_b$ is the collection of oriented circle and arcs obtained from $\Gamma$ by erasing all edges $e$ of $\Gamma$ such that $\#(c(e) \cap b) \neq 1$ and reversing the orientation of every remaining edge $e$ such  that $b^- \in c(e)$. The \emph{degree} $d(\Gamma,c)_b$ of the state $(\Gamma,c)_b$ is defined by:
\[
d(\Gamma,c)_b:= C_+ - C_- + \frac12(TR - TL - BR +BL)
\]
 where:
 \begin{itemize}
 \item $C_+$ and $C_-$ are the numbers or positively and negatively oriented circles.
 \item $TR$ is the number of arcs with both ends on the top (i. e. in $\RR\times \{1\}$) rightwards oriented.
 \item $TL$ is the number of arcs with both ends on the top leftwards oriented.
 \item $BR$ is the number of arcs with both ends on the bottom (i. e. in $\RR\times \{0\}$) rightwards oriented.
 \item $TL$ is the number of arcs with both ends on the bottom leftwards oriented.
 \end{itemize}
The \emph{degree} of a colored web $(\Gamma,c)$ is equal to $d(\Gamma,c):=\sum_{b \textrm{ bicolor}} d(\Gamma,c)_b$.

If $\Gamma$ is an open web and $c_\partial$ a coloring of its boundary, we define:
\[
\kupcs{\Gamma}^{c_\partial} :=  \sum_{\substack{c \textrm{ coloring of } \Gamma \\ c \textrm{ induces $c_\partial$ on the boundary}}} q^{d(\Gamma,c)}. 
\]
\end{dfn}


\begin{lem}
  \label{lem:additivityofdegree}
  Let $(\Gamma,c)$ be a colored open web obtained by stacking $(\Gamma', c')$ and $(\Gamma'',c'')$ one onto the other. For every bicolor $b$, we have $d((\Gamma,c)_b)=d((\Gamma',c')_b)+d((\Gamma'',c'')_b)$. It follows that we have: $d(\Gamma,c)=d(\Gamma',c')+d(\Gamma'',c'').$
\end{lem}
\begin{proof}[Sketch of the proof.]
  If we fix a bicolor $b$, the degree $d((\Gamma,c)_b)$ can be seen as a the integral of a (normalized) curvature along the arcs. With this point of view, the lemma simply follows from the Chasles relation.
\end{proof}

\begin{rmk}
From Lemma~\ref{lem:additivityofdegree}, we deduce that in order to check that $\kupcs{\cdot}$ satisfies the relations (\ref{eq:relcircle}) to (\ref{eq:relsquare3}), it is enough to check that $\kupcs{\cdot}^{c_\partial}$ satisfies the relations  (\ref{eq:relcircle}) to (\ref{eq:relsquare3}) and for every coloring $c_\partial$ of the common\label{page:footnotecolbdy}\footnote{All the webs involved in a local relation have the same boundary.}  boundary.
\end{rmk}

\section{Partitions and $q$-identities}
\label{sec:part-q-ident}
The aim of this section is to introduce degrees of partitions of ordered set and to prove Lemma~\ref{lem:key} from which Theorem~\ref{thm:evalcol} will follow. 

\begin{dfn}
  \label{dfn:degreepartition}
  Let $(X,<)$ be a finite totally ordered set and $Y$ and $Z$ two disjoint subsets of $X$. The \emph{degree $d(Y\sqcup Z)$ of $Y\sqcup Z$} is the integer defined by the formula:
\[
d(Y\sqcup Z) = \#\{(y, z)\in (Y\times Z) | y<z\} -  \#\{(y,z)\in (Y\times Z) | y>z\}.
\]
\end{dfn}

\begin{lem}
  \label{lem:binomrelation}
Let $n$ and $m$ be two integers such that $m+n\geq 1$. The following relation holds:
\[\qbin{m}{n} = q^{+m} \qbinb{m+n-1}{m}{n-1} + q^{-n}\qbinb{m+n-1}{m-1}{n}
\]
\end{lem}

\begin{proof} If $m$ or $n$ is negative, the relation reads $0=0$. We suppose that $m$ and $n$ are non-negative.
  This can be thought of in terms of degree of partition. The left-hand side counts with degree all the partitions of $\llbracket 1, m+n\rrbracket$ which consists of a set $Y$ of $m$ elements and a set $Z$ of $n$ elements. The first term of the right-hand side counts partitions such that $1$ is in $Y$, the second counts partitions such that $1$ is in $Z$. We can as well prove this equality directly:
  \begin{align*}
\qbin{m}{n}
&=  \frac{[m+n]!}{[m]! [n!]} 
 =  \frac{(q^{+m}[n] + q^{-n}[m])[m+n-1]!}{[m]! [n!]} \\
&=  \frac{(q^{+m}[n])[m+n-1]!}{[m]! [n]!} +
    \frac{(q^{-n}[m])[m+n-1]!}{[m]! [n]!} \\
&=  \frac{(q^{+m})[m+n-1]!}{[m]! [n-1]!} +
    \frac{(q^{-n})[m+n-1]!}{[m-1]! [n]!} \\
&=   q^{+m} \qbinb{m+n-1}{m}{n-1} + q^{-n}\qbinb{m+n-1}{m-1}{n}
  \end{align*}
\end{proof}
Note that $\qbin{m}{n}$ is entirely determined by the formula of Lemma~\ref{lem:binomrelation} and the fact that that for all $k\geq 0$, $\qbinb{k}{0}{k} = \qbinb{k}{k}{0} =1$.

The following lemma is not strictly necessary, however we do think that it enlightens the relation between degree of disjoint union and quantum binomials.

\begin{lem}
  \label{lem:degpartbinom}
  We consider $(X,<)$ a finite totally ordered set with $m+n$ element. Let $\mathcal{P}_{m,n}(X)$ the set of partition $Y\sqcup Z$ of $X$ such that $\#Y=m$ and $\#Z=n$. The following relation holds:
\[
\sum_{Y\sqcup Z \in \mathcal{P}_{m,n}(X)} q^{d(Y\sqcup Z)}= \qbin{m}{n}
\]
\end{lem}
\begin{proof}
  The statement actually does not depends on $X$. Hence we may suppose that $X=[1,n+m]$ with the natural order. Let us write:
\[p_{m,n}= \sum_{Y\sqcup Z \in \mathcal{P}_{m,n}([1, m+n])} q^{d(Y\sqcup Z)}= \qbin{m}{n}
\]
For every positive integer $k$, we have $p_{k,0} = p_{0,k}=1$. We have:
\begin{align*}
  p_{m+1,n+1}&= \sum_{Y\sqcup Z \in \mathcal{P}_{m+1,n+1}([1, m+n+2])} q^{d(Y\sqcup Z)}
\\ &= \sum_{\substack{Y\sqcup Z \in \mathcal{P}_{m+1,n+1}([1, m+n+2]) \\ m+n+2 \in Y}} 
 q^{d(Y\sqcup Z)} +  \sum_{\substack{Y\sqcup Z \in \mathcal{P}_{m+1,n+1}([1, m+n+2]) \\ m+n+2 \in Z}} q^{d(Y\sqcup Z)}
\\ &= \sum_{\substack{Y\sqcup Z \in \mathcal{P}_{m,n+1}([1, m+n+1])}} 
 q^{d(Y\sqcup Z)-(n+1)} +  \sum_{\substack{Y\sqcup Z \in \mathcal{P}_{m+1,n}([1, m+n+1]) \\ m+n+1 \in Y}} q^{d(Y\sqcup Z)+ m+1}
\\ &= q^{n+1}p_{m,n+1} + q^{m+1}p_{m+1,n}
\end{align*}
It satisfies the same recursion formula so the quantum binomial (and have the same initial values). This proves that for all $m$ and $n$ we have $p_{m,n}= \qbin{m}{n}$
\end{proof}

The following observation will be very useful for  proving  Lemma~\ref{lem:key}. 
\begin{lem}\label{lem:degpartsep}
  Let $X$ and $Y$ be two disjoint subsets of $\llbracket 1, M \rrbracket$ and $k$ be an integer of $\llbracket 1, M-1\rrbracket$. Let us write  $X_1=X\cap \llbracket 1,k \rrbracket$, $Y_1=Y\cap \llbracket 1,k \rrbracket$, $X_1=X\cap \llbracket k+1, N \rrbracket$ and $Y_1=Y\cap \llbracket k+1, M \rrbracket$. The following relation holds:
\[
d(X\sqcup Y) = d(X_1\sqcup Y_1) + d(X_2\sqcup Y_2) + \#X_1\#Y_2 - \#Y_1\#X_2.
\]
\end{lem}
\begin{proof}
  It follows from the definition.
\end{proof}
The following lemma is the key ingredient to prove theorem~\ref{thm:evalcol}. It should be compared to \cite[Proof of relation 2.10]{MR3263166}. 
   \begin{lem}\label{lem:key}Let us fix $X$ and $Y$ two disjoint subsets of in $ \llbracket 1,M \rrbracket$, such that $\#X= \#Y+l$ with $l\geq 0$. For every integer $k_1$, the following relation holds:
\begin{align*}
\sum_{\substack{X= X_1 \sqcup X_2 \\ \#X_1 = k_1}}q^{d(X_1\sqcup X_2)+ d(Y\sqcup X_1)}
 = \sum_{j_2= 0}^{\infty}\qbinb{l}{k_1 - j_2}{l-k_1+j_2}\sum_{\substack{Y= Y_1 \sqcup Y_2 \\ \#Y_2 = j_2}}q^{ d(Y_1\sqcup Y_2) + d(Y_2\sqcup X)}
\end{align*}
   \end{lem}
   \begin{proof}
     The proof is done by induction on the cardinal of $\#X + \# Y$. If $\#X +\#Y=0$, the relation reads $1=1$. For the induction, we need to be careful and stay in the case where $\#X \geq \#Y$. Suppose that $\#X -\#Y \geq 1$, then removing the smallest element of $X\sqcup Y$ gives us two sets $X'$ and $Y'$ such that $\#X' -\#Y' \geq 0$. 
Let us now consider the extreme case, where  $\#X = \# Y \geq 1$. We can distinguish two situations:
\begin{itemize}
\item The lowest (or the highest) element of $X\sqcup Y$ is in $Y$, then removing this element gives us two sets $X'$ and $Y'$ such that $\#X -\#Y = 1$.
\item The lowest element and the highest element of $X\sqcup Y$ are in $Y$, in this case we can find an element $k$ in $\llbracket 1,M \rrbracket$ such that, if we define $X':= X\cap \llbracket 1,k \rrbracket$, $Y':= Y\cap \llbracket 1,k \rrbracket$, $X'':= X\cap \llbracket k+1,M \rrbracket$, $Y'':= Y\cap \llbracket k+1,M \rrbracket$, we have $\#X' = \#Y'$, $\#X'' = \#Y''$ and none of these sets is empty. 
\end{itemize}
We first suppose that $\#X -\#Y=l \geq 1$ and the lowest element $t$ of $X\sqcup Y$ is in $X$. Let us write $X':= X\setminus \{t\}$.  We have:
\begin{align*} 
&\sum_{\substack{X= X_1 \sqcup X_2 \\ \#X_1 = k_1}}q^{d(X_1\sqcup X_2)+ d(Y\sqcup X_1)} \\
&\qquad = \sum_{\substack{X= X_1 \sqcup X_2 \\ \#X_1 = k_1 \\ t\in X_1}}q^{d(X_1\sqcup X_2)+ d(Y\sqcup X_1)} + \sum_{\substack{X= X_1 \sqcup X_2 \\ \#X_1 = k_1 \\ t\in X_2}}q^{d(X_1\sqcup X_2)+ d(Y\sqcup X_1)} \\
&\qquad = \sum_{\substack{X'= X'_1 \sqcup X'_2 \\ \#X'_1 = k_1 -1 }}q^{d(X'_1\sqcup X'_2)+ d(Y\sqcup X'_1) + (\#X -k_1 ) - (\#X -l)} + \sum_{\substack{X'= X'_1 \sqcup X'_2 \\ \#X_1 = k_1}}q^{d(X'_1\sqcup X'_2)+ d(Y\sqcup X'_1) - k_1} \\
&\qquad = \sum_{j_2= 0}^{\infty}q^{l-k_{1}}\qbinb{l-1}{k_1 -1 - j_2}{l-1-(k_1-1)+j_2}\sum_{\substack{Y= Y_1 \sqcup Y_2 \\ \#Y_2 = j_2}}q^{ d(Y_1\sqcup Y_2) + d(Y_2\sqcup X')} \\
&\qquad \qquad +  \sum_{j_2= 0}^{\infty}q^{-k_1}\qbinb{l-1}{k_1 - j_2}{l-1-k_1+j_2}\sum_{\substack{Y= Y_1 \sqcup Y_2 \\ \#Y_2 = j_2}}q^{ d(Y_1\sqcup Y_2) + d(Y_2\sqcup X')} \\
&\qquad = \sum_{j_2= 0}^{\infty}\left(q^{l-k_1}\qbinb{l-1}{k_1 -1 - j_2}{l-k_1+j_2} + q^{-k_1}\qbinb{l-1}{k_1 - j_2}{l-1-k_1+j_2} \right)\\
&\qquad \qquad \quad \cdot\sum_{\substack{Y= Y_1 \sqcup Y_2 \\ \#Y_2 = j_2}}q^{ d(Y_1\sqcup Y_2) + d(Y_2\sqcup X')} \\
&\qquad = \sum_{j_2= 0}^{\infty}q^{-j_2}\qbinb{l}{k_1 - j_2}{l-k_1+j_2} \sum_{\substack{Y= Y_1 \sqcup Y_2 \\ \#Y_2 = j_2}}q^{ d(Y_1\sqcup Y_2) + d(Y_2\sqcup X')}\\
&\qquad = \sum_{j_2= 0}^{\infty}\qbinb{l}{k_1 - j_2}{l-k_1+j_2} \sum_{\substack{Y= Y_1 \sqcup Y_2 \\ \#Y_2 = j_2}}q^{ d(Y_1\sqcup Y_2) + d(Y_2\sqcup X)}
\end{align*}
The case were the greatest element of $X\sqcup Y$ is in $X$ is analogue.
We suppose now that $\#X -\#Y=l \geq 0$ and the lowest element $t$ of $X\sqcup Y$ is in $Y$. Let us write $Y':= Y\setminus \{t\}$.  We have:
\begin{align*}
&\sum_{j_2= 0}^{\infty}\qbinb{l}{k_1 - j_2}{l-k_1+j_2}\sum_{\substack{Y= Y_1 \sqcup Y_2 \\ \#Y_2 = j_2}}q^{ d(Y_1\sqcup Y_2) + d(Y_2\sqcup X)} \\
&\qquad = \sum_{j_2= 0}^{\infty}\qbinb{l}{k_1 - j_2}{l-k_1+j_2}\sum_{\substack{Y= Y_1 \sqcup Y_2 \\ \#Y_2 = j_2 \\ t\in Y_1}}q^{ d(Y_1\sqcup Y_2) + d(Y_2\sqcup X)} \\
&\qquad \qquad +  \sum_{j_2= 0}^{\infty}\qbinb{l}{k_1 - j_2}{l-k_1+j_2}\sum_{\substack{Y= Y_1 \sqcup Y_2 \\ \#Y_2 = j_2 \\ t\in Y_2}}q^{ d(Y_1\sqcup Y_2) + d(Y_2\sqcup X)} \\
&\qquad = \sum_{j_2= 0}^{\infty}\qbinb{l}{k_1 - j_2}{l-k_1+j_2}\sum_{\substack{Y'= Y'_1 \sqcup Y'_2 \\ \#Y'_2 = j_2 }}q^{ d(Y'_1\sqcup Y'_2) + d(Y'_2\sqcup X) + j_2} \\
&\qquad \qquad +  \sum_{j_2= 0}^{\infty}\qbinb{l}{k_1 - j_2}{l-k_1+j_2}\sum_{\substack{Y= Y'_1 \sqcup Y'_2 \\ \#Y'_2 = j_2 -1 }}q^{ d(Y'_1\sqcup Y'_2) + d(Y'_2\sqcup X) - (\#X - l -j_2) +\# X} \\
&\qquad = \sum_{j_2= 0}^{\infty}\qbinb{l}{k_1 - j_2}{l-k_1+j_2}\sum_{\substack{Y'= Y'_1 \sqcup Y'_2 \\ \#Y'_2 = j_2 }}q^{ d(Y'_1\sqcup Y'_2) + d(Y'_2\sqcup X) + j_2} \\
&\qquad \qquad +  \sum_{j_2= 0}^{\infty}\qbinb{l}{k_1 - (j_2+1)}{l-k_1+(j_2+1)}\sum_{\substack{Y'= Y'_1 \sqcup Y'_2 \\ \#Y'_2 = j_2 }}q^{ d(Y'_1\sqcup Y'_2) + d(Y'_2\sqcup X) +l +j_2+1} \\
&\qquad = \sum_{j_2= 0}^{\infty}\left(q^{j_2}\qbinb{l}{k_1 - j_2}{l-k_1+j_2} +  q^{l+j_2+1 }\qbinb{l}{k_1 - (j_2+1)}{l-k_1+(j_2+1)} \right) \\
&\qquad \qquad \qquad \cdot \sum_{\substack{Y= Y_1 \sqcup Y_2 \\ \#Y'_2 = j_2 }}q^{ d(Y_1\sqcup Y_2) + d(Y_2\sqcup X) + l +j_2+1} \\
&\qquad = \sum_{j_2= 0}^{\infty}q^{k_1}\qbinb{l+1}{k_1 - j_2}{l+1-k_1+j_2} \sum_{\substack{Y= Y_1 \sqcup Y_2 \\ \#Y'_2 = j_2 }}q^{ d(Y_1\sqcup Y_2) + d(Y_2\sqcup X)+  l +j_2+1} \\
&\quad = \sum_{\substack{X= X_1 \sqcup X_2 \\ \#X_1 = k_1}}q^{d(X_1\sqcup X_2)+ d(Y'\sqcup X_1)+k_1}\\
&\quad = \sum_{\substack{X= X_1 \sqcup X_2 \\ \#X_1 = k_1}}q^{d(X_1\sqcup X_2)+ d(Y\sqcup X_1)}
\end{align*}
The case were the greatest element of $X\sqcup Y$ is in $Y$ is analogue.

We suppose now that the greatest and the lowest element of $X\sqcup Y$ are in $X$ and that $l=0$. We use the notations explained before and we write $k'= \#X' = \# Y'$ and $k''= \#X'' = \# Y''$. For readability it is convenient to set $l'=l''= 0$. We have:
\begin{align*}
&\sum_{j_2= 0}^{\infty}\qbinb{l}{k_1 - j_2}{l-k_1+j_2}\sum_{\substack{Y= Y_1 \sqcup Y_2 \\ \#Y_2 = j_2}}q^{ d(Y_1\sqcup Y_2) + d(Y_2\sqcup X)}  = \sum_{\substack{Y= Y_1 \sqcup Y_2 \\ \#Y_2 = k_1}}q^{ d(Y_1\sqcup Y_2) + d(Y_2\sqcup X)} \\
&= \sum_{k'_1 + k''_1 = k_1}\sum_{\substack{Y'= Y'_1 \sqcup Y'_2 \\ \#Y'_2 = k'_1 \\Y''=Y''_1 \sqcup Y''_2 \\ \#Y''_2 = k''_1 }}q^{ d(Y'_1\sqcup Y'_2) + d(Y'_2\sqcup X') + d(Y''_1\sqcup Y''_2) + d(Y''_2\sqcup X'') +(k'-k'_1) k''_1 - k'_1(k''-k''_1) + k'_1k'' - k'k''_1} \\
& = \sum_{k'_1 + k''_1 = k_1}\sum_{\substack{Y'= Y'_1 \sqcup Y'_2 \\ \#Y'_2 = k'_1 \\Y''=Y''_1 \sqcup Y''_2 \\ \#Y''_2 = k''_1 }}q^{ d(Y'_1\sqcup Y'_2) + d(Y'_2\sqcup X') + d(Y''_1\sqcup Y''_2) + d(Y''_2\sqcup X'')} \\
& = \sum_{k'_1 + k''_1 = k_1}\left(\sum_{\substack{Y'= Y'_1 \sqcup Y'_2 \\ \#Y'_2 = k'_1}} q^{ d(Y'_1\sqcup Y'_2) + d(Y'_2\sqcup X')} \right)\left( \sum_{\substack{\\Y''=Y''_1 \sqcup Y''_2 \\ \#Y''_2 = k''_1 }}q^{ d(Y''_1\sqcup Y''_2) + d(Y''_2\sqcup X'')} \right) \\
& = \sum_{k'_1 + k''_1 = k_1}\left(\sum_{j'_2=0}^\infty \qbinb{l'}{k'_1-j'_2}{ l'- k'_1 +j'_2}\sum_{\substack{Y'= Y'_1 \sqcup Y'_2 \\ \#Y'_2 = j'_2}}q^{ d(Y'_1\sqcup Y'_2) + d(Y'_2\sqcup X')} \right)\\
& \qquad \left(\sum_{j''_2=0}^\infty \qbinb{l''}{k''_1-j''_2}{ l''- k''_1 +j''_2} \sum_{\substack{\\Y''=Y''_1 \sqcup Y''_2 \\ \#Y''_2 = j''_2 }}q^{d(Y''_1\sqcup Y''_2) + d(Y''_2\sqcup X'')} \right)\\
& = \sum_{k'_1 + k''_1 = k_1}\left(\sum_{\substack{X= X'_1 \sqcup X'_2 \\ \#X'_1 = k'_1}}q^{d(X'_1\sqcup X'_2)+ d(Y'\sqcup X'_1)} \right)\left( \sum_{\substack{X''= X''_1 \sqcup X''_2 \\ \#X''_1 = k''_1}}q^{d(X''_1\sqcup X''_2)+ d(Y''\sqcup X''_1)} \right) \\
& = \sum_{k'_1 + k''_1 = k_1}\sum_{\substack{X= X'_1 \sqcup X'_2 \\ \#X'_1 = k'_1\\ X''= X''_1 \sqcup X''_2 \\ \#X''_1 = k''_1}}q^{d(X'_1\sqcup X'_2)+ d(Y'\sqcup X'_1) + d(X''_1\sqcup X''_2)+ d(Y''\sqcup X''_1)} \\
& = \sum_{k'_1 + k''_1 = k_1}\sum_{\substack{X= X'_1 \sqcup X'_2 \\ \#X'_1 = k'_1\\ X''= X''_1 \sqcup X''_2 \\ \#X''_1 = k''_1}}q^{d(X'_1\sqcup X'_2)+ d(Y'\sqcup X'_1) + d(X''_1\sqcup X''_2)+ d(Y''\sqcup X''_1)+ k'_1 (k''-k''_1) -(k'-k'_1)k''_1 +  k'k''_1 -k'_1k''} \\
& = \sum_{\substack{X= X_1 \sqcup X_2 \\ \#X_1 = k_1}}q^{d(X_1\sqcup X_2)+ d(Y\sqcup X_1)} 
\end{align*}
   \end{proof}

\section{Checking the skein relations}
\label{sec:check-skein-relat}

\begin{lem}
  Let $\Gamma$ be a web, and $\Gamma'$ the web obtained by deleting all the edges labelled by $0$ or $N$. We have:
\[ \kupcs{\Gamma} =\kupcs{\Gamma'}.\]
\end{lem}
\begin{proof}
  We clearly have a one-one correspondence between the colorings of $\Gamma$ and of $\Gamma'$ since an edge labeled $0$ can only be colored by $\varnothing$ and an edge labeled $N$ can only be colored by $\cN$. Furthermore, the edges of $\Gamma$ labeled by $0$ and $N$ never appear in any state for any coloring, this means that the state $(\Gamma,c)_b$ and $(\Gamma', c')_b$ are equal (where $c$ and $c'$ are two  colorings corresponding one to another). This proves that $\kupcs{\Gamma} =\kupcs{\Gamma'}$. 
\end{proof}

\begin{lem}
  It is enough to check relations (\ref{eq:relass}) and (\ref{eq:relsquare3}).
\end{lem}

\begin{proof}
We prove that (\ref{eq:relcircle}) and (\ref{eq:relbin2}) follow from (\ref{eq:relbin1}). We suppose that relation (\ref{eq:relbin1}) holds. 
 We have
\[   \kupcs{\vcenter{\hbox{\tikz[scale= 0.5]{\draw[->] (0,0) arc(0:360:1cm) node[right] {\small{$\!k\!$}};}}}}= 
\kupc{\digonNk} = \qbinb{N}{k}{N-k}.
 \]
This proves relation (\ref{eq:relcircle}) holds.
\begin{align*}
\kupc{\digonb} &= \kupc{\digonc} = \qbinb{N-m}{n}{N-m -n} \kupc{\vertc} \\&= \qbinb{N-m}{n}{ N-m -n} \kupc{\vertb}
\end{align*}
This proves relation (\ref{eq:relbin2}) holds.

We prove that relations (\ref{eq:relbin1}), (\ref{eq:relsquare1}) and (\ref{eq:relsquare2}) follow from (\ref{eq:relsquare3}). We suppose that (\ref{eq:relsquare3}) holds. 

Relation (\ref{eq:relsquare2}) is a special case of relation (\ref{eq:relsquare3}): by setting $n= 1$, $m= l'$, $k= l'+n' -1$ and $l= m'-1$ in (\ref{eq:relsquare3}), we obtain (\ref{eq:relsquare2})  with all the labels replaced by labels with~$'$.

Relation (\ref{eq:relsquare1}) is a special case of relation (\ref{eq:relsquare1}): by setting $l= 1$, $m= N-m'$, and $n= N-m'-1$ in (\ref{eq:relsquare2}), we obtain (\ref{eq:relsquare1})  with all the labels replaced by labels with~$'$.

Relation (\ref{eq:relbin1}) is a special case of relation (\ref{eq:relsquare3}): by setting $m=n= 0$, $l= m'+n'$  and $k= m'$ in (\ref{eq:relsquare3}), we obtain (\ref{eq:relbin1})  with all the labels replaced by labels with~$'$.
\end{proof}

\begin{lem}
  \label{lem:associativity}
  The following relation holds:
  \[\kupc{\stgamma}^{c_\partial} = \kupc{\stgammaprime}^{c_\partial}\]
where $c_\partial$ is any coloring of the boundary (see footnote on page~\pageref{page:footnotecolbdy}).
\end{lem}
\begin{proof}

    Let $\Gamma$ be the open web on the left and $\Gamma'$ the open web on the right. The boundary consists of four points, $\tau_1, \tau_2$ and $\tau_3$ on the top and $\beta$ on the bottom. Let us consider a coloring $c_\beta$ of the boundary (that is any application $\{\tau_1, \tau_2, \tau_3, \beta \} \to \cN$ such that $\#c_\partial(\tau_1)= i$, $\#c_\partial(\tau_2)= j$, $\#c_\partial(\tau_3)= k$ and $\#c_\partial(\beta)= i+j+k$. 

Suppose $X_1:=c_\partial(\tau_1)$,  $X_2:=c_\partial(\tau_2)$ and $X_3:=c_\partial(\tau_3)$  form a partition of $c_\partial(\beta)$. Then there exists a unique coloring $c$ of $\Gamma$ and a unique coloring $c'$ of $\Gamma'$ compatible with $c_\partial$. We need to compare the degree $d(\Gamma,c)$ and $d(\Gamma', c')$. The values of  $d(\Gamma,c)_b$ and $d(\Gamma', c')_b$ for every bicolor $b$ are given in Table~\ref{table:associ}. We have
$d(\Gamma,c)_b =d(\Gamma', c')_b$ for all $b$. Hence $d(\Gamma,c) =d(\Gamma', c')$ and $\kupcs{\Gamma}^{c_\partial} =
\kupcs{\Gamma'}^{c_\partial}$.

Suppose $c_\partial(\tau_1)$,  $c_\partial(\tau_2)$ and $c_\partial(\tau_3)$  do not form a partition of $c_\partial(\beta)$. Then there is no coloring of $\Gamma$ inducing $c_\partial$ and no coloring of $\Gamma'$ inducing $c_\partial$. Hence we have $\kupcs{\Gamma}^{c_\partial} =
\kupcs{\Gamma'}^{c_\partial}=0$. 

Finally we have   \[\kupc{\stgamma}^{c_\partial} = \kupc{\stgammaprime}^{c_\partial}\]
for all colorings $c_\partial$ of the boundary.
\newcommand{\x}{\cellcolor{blue!25}}
\begin{table}[ht]
  \centering
  \begin{tabular}{|c|c|c|c|c|}
    \hline
    $\!\!b_-\backslash b_+\!\!$ & $X_1$        & $X_2$      & $ X_3$       & $R$      \\ \hline
    $X_1$               & \stnoa       & \stonetwoa & \stonethreea & \stdonea  \\ \hline
    $X_2$               & \sttwoonea   & \stnoa     & \sttwothreea & \stdtwoa  \\ \hline
    $X_3$               & \x\stthreeonea & \stthreetwoa & \stnoa       & \stdthreea \\ \hline
    $R$                 & \stonea     & \sttwoa   & \stthreea   & \stnoa \\ \hline
  \end{tabular}\quad 
  \begin{tabular}{|c|c|c|c|c|}
    \hline
    $\!\!b_-\backslash b_+\!\!$ & $X_1$        & $X_2$      & $ X_3$       & $R$      \\ \hline
    $X_1$               & \stnob       & \stonetwob & \stonethreeb & \stdoneb  \\ \hline
    $X_2$               & \sttwooneb   & \stnob     & \sttwothreeb & \stdtwob  \\ \hline
    $X_3$               & \stthreeoneb & \stthreetwob & \stnob       & \stdthreeb \\ \hline
    $R$                 & \stoneb     & \sttwob   & \stthreeb   & \stnob \\ \hline
  \end{tabular}

\caption{Computations of $d(\Gamma, c)_b$ (on the left) and $d(\Gamma', c')_b$ (on the right). How to read these tables: The diagram contained in the blue cell, is the state $(\Gamma,c)_{b}$ when $b_+$ is in $X_1$ and $b_-$ is in $X_3$. The number $+\frac12$ contained in the blue cell is $d(\Gamma,c)_b$. How to check these tables: There is an anti-symmetry property: if one transposed the table, one should obtain the same diagram with the orientation reversed and the degree should be multiplied by $-1$. There is a consistency property for each line and each column: If two diagrams are on the same line or on the same column and have a common solid edge, the orientations of this edge on the two diagrams should be the same.}
\label{table:associ}
\end{table}

\end{proof}

 \begin{lem}
   \label{lem:square}
   The following relation holds:
\[
\kupc{\squarec}^{c_\partial}\!\!\! = \sum_{j=0}^\infty\qbinb{l}{k-j}{l-k+j} \kupc{\squared}^{c_\partial}
\]
where $c_\partial$ is any coloring of the boundary. 
 \end{lem}
 \begin{proof}
Let us denote by $\Gamma$ the open web on the left-hand side of the relation and by $(\Gamma_j)_{j\in \NN}$ the open webs on the right-hand side of the relation. 

Let $c$ be a coloring of $\Gamma$ and $i$ an element of $\cN$. We consider the set $E_i$ of edges $e$ of $\Gamma$ such that $i$ is in $c(e)$. Due to the flow condition on MOY graphs (see Remark~\ref{rmk:flow}), the following configurations $E_i$ are the only possible ones (the solid edges are in $E_i$ the others not):
\[
\twovertupupNAME{C},  \quad \zigzagrightupNAME{X_1}, \quad\onevertemptyupNAME{X_2}, \quad \onevertupemptyNAME{Y}, 
 \quad  \bottomlefttoprighthighNAME{Z}, \quad  \bottomrighttopleftlowNAME{T} \quad\textrm{and}\quad  \emptysquareNAME{R}.
\]

This gives us a partition of $\cN$ into 7 sets : $C$, $X_1$, $X_2$, $Y$, $Z$, $T$ and $R$. The coloring of the boundary induced by $c$ is:
\begin{itemize}
\item $C\sqcup Y \sqcup Z$ for the top right point,
\item $C\sqcup X_1 \sqcup X_2 \sqcup T$ for the top left point,
\item $C\sqcup Y \sqcup Z$ for the bottom right point,
\item $C\sqcup X_1\sqcup X_2 \sqcup Z$ for the bottom left point.
\end{itemize}
We have the following conditions on the cardinals of the sets $C$, $X_1$, $X_2$, $Y$, $Z$, $T$ and $R$ (we name them $c$, $x_1$, $x_2$, $y$, $z$, $t$ and $r$):
\begin{itemize}
\item $c+y+z = m$,
\item $c+y+t = n$,
\item $c+x_1 + x_2+t = n+l$,
\item $c+x_1 + x_2 +z = m+l $,
\item $x_1 + t = k $,
\item $c+r+ x_1 + x_2 + y+z+t+r =N$.
\end{itemize}
From this we easily deduce that we have:
\[
x_1 + x_2 -y =l,  \quad  x_1+x_2= n+l -c-t\quad \textrm{and} \quad  x_1 = k-t.
\]
If we are given a coloring of the boundary (which extends to a coloring of $\Gamma$), we can recover\footnote{For example the set $C$ is the intersection of the colorings of the four points and the set $X$ is the intersection of the coloring of the two points on the right minus the set $C$. }
$C$, $R$, $Y$, $Z$, $T$ and $X:=X_1\sqcup X_2$.

On the other hand, if we are given a partition of $\cN$ into $6$ sets $C$, $X$, $Y$, $Z$, $T$ and $R$ such that there cardinals $c$, $x$, $y$, $z$, $t$ and $r$ satisfy:
\begin{itemize}
\item $c+y+z = m$,
\item $c+y+t = n$,
\item $c+x+t = n+l$,
\item $c+x +z = m+l $
\item and $t \leq k $,
\end{itemize}
every partition of $X$ into two sets $X_1$ and $X_2$, such that $\#X_1 = k-t$ provides a coloring of $\Gamma$. This means, that there exist exactly $
\begin{pmatrix}
 n+l-c-t \\
k-t 
\end{pmatrix}
$ such colorings.

We now consider a coloring $c_j$ of the open web $\Gamma_j$. As we did for $\Gamma$, we can form a partition of $\cN$:

\[
\twovertupupNAME{C}, \quad   \onevertemptyupNAME{X}, \quad \onevertupemptyNAME{Y_1},  \quad 
 \zigzagleftupNAME{Y_2}, \quad, \bottomlefttoprightlowNAME{Z}, \quad  \bottomrighttoplefthighNAME{T} \quad\textrm{and}\quad   \emptysquareNAME{R}.
\]

The coloring of the boundary induced by $c_j$ is:
\begin{itemize}
\item $C\sqcup Y_1\sqcup Y_2 \sqcup Z$ for the top right point,
\item $C\sqcup X \sqcup T$ for the top left point,
\item $C\sqcup Y_1\sqcup Y_2 \sqcup Z$ for the bottom right point,
\item $C\sqcup X \sqcup Z$ for the bottom left point.
\end{itemize}
We have the following conditions on the cardinals of the sets $C$, $X$, $Y_1$, $Y_2$, $Z$, $T$ and $R$ (we name them $c$, $x$, $y_1$, $y_2$, $z$, $t$ and $r$):
\begin{itemize}
\item $c+y_1+ y_2+z = m$,
\item $c+y_1+y_2+t = n$,
\item $c+x+t = n+l$,
\item $c+x +z = m+l $,
\item $y_2 + t = j $,
\item $c+r+ x_1 + x_2 + y+z+t+r =N$.
\end{itemize}
From this we easily deduce that we have:
\[
x - (y_1 +y_2) =l,  \quad  x = n+l -c-t\quad \textrm{and} \quad  y_1 = j-t.
\]

If we are given a coloring of the boundary (which extend to a coloring of $\Gamma_j$), we can recover 
$C$, $X$, $Z$, $T$, $R$ and $Y:=Y_1\sqcup Y_2$ with the same strategy as for $\Gamma$. 

If we are given a partition of $\cN$ into six sets $C$, $X$, $Y$, $Z$, $T$ and $R$ such that their cardinals $c$, $x$, $y$, $z$, $t$ and $r$ satisfy 
\begin{itemize}
\item $c+y_1+ y_2+z = m$,
\item $c+y_1+y_2+t = n$,
\item $c+x+t = n+l$,
\item $c+x +z = m+l $,
\item $ t \leq j $,
\end{itemize}
every partition of $Y$ into two sets $Y_1$and $X_2$, such that $\#Y_2 = j-t$ provides a coloring of $\Gamma_j$. This means, that there exist exactly $
\begin{pmatrix}
n-c-t \\
j-t 
\end{pmatrix}
$ such colorings.

Note that if $j>k$, the coefficient multiplying $\Gamma_j$ is equal to zero. Hence we may suppose that $j\leq k$. In this case the condition for a coloring to be extendable to $\Gamma$ is stronger than the condition to be extendable to $\Gamma_j$. If a coloring $c_\partial$ does not extend to $\Gamma$, the equality simply says $0=0$.

Let us suppose that $c_\partial$ extends to $\Gamma$. We denote $C$, $R$, $X$, $Z$, $T$, $Y$ the partition of $\cN$ such that:
\begin{itemize}
\item  The top right point is colored by $C\sqcup Y \sqcup Z$,
\item  The top left point is colored by $C\sqcup X \sqcup T$,
\item  The bottom right point is colored by $C\sqcup Y \sqcup T$,
\item  The bottom left point  is colored by $C\sqcup X \sqcup Z$.
\end{itemize}
We have $\#X= \#Y +l$.

The colorings of $\Gamma$ which induce $c_\partial$ on the boundary are given by partitions $X_1\sqcup X_2$ of $X$ such that $X_1$ has $k-t$ elements. If we fix $c$ such a coloring (notations are given in Figure~\ref{fig:colgammanot}), we can compute $d((\Gamma,c)_b)$ for every bicolor $b$. The computations are done in Table~\ref{tab:leqsquare3type3a}. From the table we deduce that $d((\Gamma,c)_b)=\delta(c_\partial) + d(X_1\sqcup X_2) + d(Y\sqcup X_1)$ where $\delta(c_\partial)$ is a constant depending only on $c_\partial$. We obtain:
\[
\kupcs{\Gamma}^{c_\partial} = q^{\delta(c_\partial)}\sum_{\substack{X=X_1\sqcup X_2} \\ \#X_1 = k-t} q^{d(X_1\sqcup X_2) + d(Y\sqcup X_1)}.
\]

\begin{figure}[ht!]
  \centering
  \begin{tikzpicture}
    \begin{scope}
\coordinate (B1) at (-1,0);
\coordinate (B2) at (1,0);
\coordinate (C1) at (-1,1);
\coordinate (D1) at (-1,2);
\coordinate (C2) at (1,1);
\coordinate (D2) at (1,2);
\coordinate (T1) at (-1,3);
\coordinate (T2) at (1,3);
\draw[->] (B1) -- (C1) node[at start, below] {\tiny{$C\sqcup Z\sqcup Y$}};
\draw[<-] (D1) -- (C1) node[midway, left   ] {\tiny{$C\sqcup X_1\sqcup Y\sqcup Z \sqcup T$}};
\draw[->] (D1) -- (T1) node[at end , above ] {\tiny{$C\sqcup T\sqcup Y$}};
\draw[<-] (C2) -- (B2) node[at end, below] {\tiny{$C\sqcup X_1\sqcup X_2\sqcup T$}};
\draw[->] (C2) -- (D2) node[midway, right] {\tiny{$C\sqcup X_2$}};
\draw[<-] (T2) -- (D2) node[at start, above] {\tiny{$C\sqcup X_1 \sqcup X_2 \sqcup Z$}};
\draw[<-] (D2) -- (D1) node[midway, above] {\tiny{$X_1 \sqcup Z$}};
\draw[<-] (C1) -- (C2) node[midway, below] {\tiny{$X_1\sqcup T$}};
\end{scope}
    \begin{scope}[xshift=5cm]
       \begin{scope}
\coordinate (B1) at (-1,0);
\coordinate (B2) at (1,0);
\coordinate (C1) at (-1,1);
\coordinate (D1) at (-1,2);
\coordinate (C2) at (1,1);
\coordinate (D2) at (1,2);
\coordinate (T1) at (-1,3);
\coordinate (T2) at (1,3);
\draw[->] (B1) -- (C1) node[at start, below] {\tiny{$C\sqcup   Y_1\sqcup Y_2 \sqcup Z$}};
\draw[<-] (D1) -- (C1) node[midway, left   ] {\tiny{$C\sqcup   Y_1$}};
\draw[->] (D1) -- (T1) node[at end , above ] {\tiny{$C\sqcup   T\sqcup Y_1$}};
\draw[<-] (C2) -- (B2) node[at end, below] {\tiny{$C\sqcup   X\sqcup T $}};
\draw[->] (C2) -- (D2) node[midway, right] {\tiny{$C\sqcup    X \sqcup Y_2 \sqcup Z \sqcup T$}};
\draw[<-] (T2) -- (D2) node[at start, above] {\tiny{$C\sqcup   X \sqcup Z $}};
\draw[->] (D2) -- (D1) node[midway, above] {\tiny{$Y_2\sqcup T$}};
\draw[->] (C1) -- (C2) node[midway, below] {\tiny{$Y_2\sqcup Z$}};
\end{scope}
    \end{scope}
  \end{tikzpicture}

  \caption{Notations for the coloring $c$ of $\Gamma$ and the coloring $c_j$ of $\Gamma_j$.}
  \label{fig:colgammanot}
\end{figure}
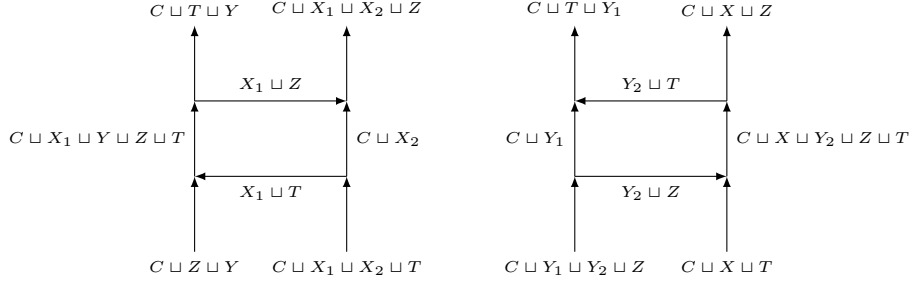

\newcommand{\z}{\cellcolor{red!25}}
\begin{table}[ht]
  \centering
  \begin{tabular}{|c|c|c|c|c|c|c|c|c|} \hline
$b^+\backslash b^-$ & $C$                       & $X_1$                & $X_2$                   & $Y$                  & $Z$                       & $T$                      & $R$                       \\ \hline 
$C$                 & $\emptysquare$            & $\zigzagleftup $     & $\onevertupempty$       & $\onevertemptyup$    & $\bottomrighttoplefthigh$ & $\bottomlefttoprightlow$ & $\twovertupup$            \\ \hline
$X_1$               & $\zigzagleftdown $        & $\emptysquare$       & \z $\squaremiddleneg$   & \z$\twohorrightleft$ & $\onehoremptyleft$        & $\onehorrightempty$      & $\zigzagrightup $         \\ \hline
$X_2$               & $\onevertdownempty$       & \z$\squaremiddlepos$ & $\emptysquare$          & $\twovertdownup$     & $\onehordeepemptyleft$    & $\onehordeeprightempty$  & $\onevertemptyup$         \\ \hline
$Y$                 & $\onevertemptydown$       & \z$\twohorleftright$ & $\twovertupdown$        & $\emptysquare$       & $\onehorleftempty$        & $\onehoremptyright$      & $\onevertupempty$         \\ \hline
$Z$                 & $\topleftbottomrighthigh$ & $\onehoremptyright$  & $\onehordeepemptyright$ & $\onehorrightempty$  & $\emptysquare$            & $\twohorrightright$      & $\bottomlefttoprighthigh$ \\ \hline
$T$                 & $\toprightbottomleftlow$  & $\onehorleftempty$   & $\onehordeepleftempty$  & $\onehoremptyleft$   & $\twohorleftleft$         & $\emptysquare$           & $\bottomrighttopleftlow$  \\ \hline
$R$                 & $\twovertdowndown$        & $\zigzagrightdown$   & $\onevertemptydown$     & $\onevertdownempty$  & $\toprightbottomlefthigh$ & $\topleftbottomrightlow$ & $\emptysquare$            \\ \hline
  \end{tabular}

  \caption{Computations of $d(\Gamma, c)_b$. The red cells emphasize the contributions which do not only depend on $c_\partial$.}
  \label{tab:leqsquare3type3a}
\end{table}
\begin{table}[ht]
  \centering
  \begin{tabular}{|c|c|c|c|c|c|c|c|c|} \hline
$b^+\backslash b^-$ & $C$                       & $X$                  & $Y_1$                   & $Y_2$                & $Z$                      & $T$                       & $R$                       \\ \hline 
$C$                 & $\emptysquare$            & $\onevertupempty$    & $\onevertemptyup$       & $\zigzagrightup$     & $\bottomrighttopleftlow$ & $\bottomlefttoprighthigh$ & $\twovertupup$            \\ \hline
$X$                 & $\onevertdownempty$       & $\emptysquare$       & $\twovertdownup$        & \z$\twohorrightleft$ & $\onehoremptyleft$       & $\onehorrightempty$       & $\onevertemptyup$         \\ \hline
$Y_1$               & $\onevertemptydown $      & $\twovertdownup$     & $\emptysquare$          & \z$\squaremiddleneg$ & $\onehordeepleftempty$   & $\onehordeepemptyright$   & $\onevertupempty$         \\  \hline
$Y_2$               & $\zigzagrightdown$        & \z$\twohorleftright$ & \z$\squaremiddlepos$    & $\emptysquare$       & $\onehorleftempty$       & $\onehoremptyright$       & $\zigzagleftup$           \\ \hline
$Z$                 & $\topleftbottomrightlow$  & $\onehoremptyright$  & $\onehordeeprightempty$ & $\onehorrightempty$  & $\emptysquare$           & $\twohorrightright$       & $\bottomlefttoprightlow$  \\\hline
$T$                 & $\toprightbottomlefthigh$ & $\onehorleftempty$   & $\onehordeepemptyleft$  & $\onehoremptyleft$   & $\twohorleftleft$        & $\emptysquare$            & $\bottomrighttoplefthigh$ \\\hline
$R$                 & $\twovertdowndown$        & $\onevertemptydown$  & $\onevertdownempty$     & $\zigzagleftdown$    & $\toprightbottomleftlow$ & $\topleftbottomrighthigh$ & $\emptysquare$            \\ \hline
  \end{tabular}

  \caption{Computations of $d(\Gamma_j, c_j)_b$  The red cells emphasize the contributions which do not only depend on $c_\partial$.}
  \label{tab:leqsquare3type3b}
\end{table}

The colorings of $\Gamma_j$ which induce $c_\partial$ on the boundary are given by partitions $Y_1\sqcup Y_2$ of $Y$ such that $Y_2$ has $j-t$ elements. If we fix $c_j$ such a coloring (notations are given in Figure~\ref{fig:colgammanot}), we can compute $d((\Gamma_j,c_j)_b)$ for every bicolor $b$. The computations are done in Table~\ref{tab:leqsquare3type3b}. From the table we deduce that $d((\Gamma_j,c_j)_b)=\delta(c_\partial) + d(Y_1\sqcup Y_2) + d(Y_2\sqcup X)$ where $\delta( c_\partial)$ is the same constant as for $d((\Gamma, c)_b$. We obtain
\[
\kupcs{\Gamma_j}^{c_\partial} = q^{\delta(c_\partial)}\sum_{\substack{Y=Y_1\sqcup Y_2} \\ \#Y_1 = j-t} q^{d(Y_1\sqcup Y_2) + d(Y_2\sqcup X)}.
\]

We conclude the proof by applying Lemma~\ref{lem:key} with $k_1 = k-t$ and $j_2= j-t$.

\end{proof}

\section{A new of skein relation}
\label{sec:new-kind-skein}

Now that we know that $\kupcs{\cdot}$ and $\kups{\cdot}$ coincide on MOY graph, we might denote both by $\kups{\cdot}_N$. We would like now to relate $\sll_N$-evaluations of MOY graphs for different $N$'s. 

\begin{dfn}
  Let $\Gamma$ be a closed MOY graph and $A=\{\alpha_1, \dots, \alpha_k\}$ a collection of disjoint oriented cycles in $\Gamma$. We denote by $\Gamma^{A_N}$ the MOY graph obtained by reversing the orientations of every edge included in $A$ and replacing the label $i$ of such an edge by $N-i$.
\end{dfn}

\begin{rmk}
  The MOY graphs $\Gamma$ and $\Gamma^{A_N}$ are equivalent as $\sll_N$-webs but not as $\sll_{N-1}$-webs (see Remark~\ref{rmk:eqwebcoloring}).
\end{rmk}

\begin{prop}
\label{prop:newskein}
  Let $\Gamma$ be a MOY graph. The following equality holds:
  \begin{align*}
\kups{\Gamma}_N &= \sum_{\textrm{$A$ collection of disjoint cycles}} q^{-w(\Gamma^{A_N})}\kups{\Gamma^{A_N}}_{N-1}  \\ &= \sum_{\textrm{$A$ collection of cycles}} q^{+w(\Gamma^A)}\kups{\Gamma^{A_N}}_{N-1}.
\end{align*}
\end{prop}

\begin{proof}
  We only prove 
\[\kups{\Gamma}_N = \sum_{\textrm{$A$ collection of cycles}} q^{-w(\Gamma^{A_N})}\kups{\Gamma^{A_N}_{N-1}} \]
since the other equality follows from the symmetry of $\kups{\cdot}$ in $q$ and $q^{-1}$. This equality is a consequence from the following observation: If $c$ is a coloring of $\Gamma$, we can consider the set $A(c)$ of disjoint cycles which consists of the edges whose colorings contain $N$. The web $\Gamma^{A(c)_N}$ inherits a coloring $c'$ such that none of the colors of the edges contain $N$. The degree $d_N$ of the coloring $c'$ as a coloring of an $\sll_{N}$-web is equal to $d_{N-1} - w(\Gamma^{A_N})$ where $d_{N-1}$ is the degree of the coloring $c'$ as a coloring of an $\sll_{N-1}$-web: 
\[
d_N - d_{N-1} = \sum_{j=1}^{N-1} d(\Gamma^{A_N},c')_{(j,N)} =  - w(\Gamma^{A_N}).
\]
Conversely, if $A$ is a collection of disjoint cycles of $\Gamma$, and $c'$ a coloring of $\Gamma^{A_N}$ as a $\sll_{N-1}$-web, the $\sll_N$-web $\Gamma$ inherits a coloring $c$ and the previous equality holds.  
\end{proof}

\bibliographystyle{alpha}
\bibliography{../../Latex/biblio}

\end{document}